\documentclass[11pt]{amsart}

\usepackage[lmargin=1in,rmargin=1in,tmargin=1in,bmargin=1in]{geometry}
\usepackage{amsmath,amsthm,amsfonts,amssymb,verbatim}
\usepackage{graphicx}
\usepackage{tabularx}
\usepackage{overpic}
\usepackage{tikz}\usetikzlibrary{decorations.markings}
\usetikzlibrary{arrows}
\usepackage{rotating}
\usepackage[colorlinks=true,citecolor=blue,linkcolor=blue]{hyperref}
\usepackage{enumitem}
\usepackage{framed}
\usepackage{xcolor}
\usepackage{placeins}
\usepackage{booktabs}
\usepackage{pinlabel}
\usepackage{picins}
\usepackage{color}
\usepackage{multirow}
\usepackage{colortbl}
\usepackage{geometry}
\usepackage{slashed}
\usepackage[font=footnotesize,labelfont=bf]{caption}
\usepackage{arydshln}
\usepackage{blkarray}
\usepackage{mathabx}

\usepackage{mathabx,epsfig}
\def\acts{\mathrel{\reflectbox{$\righttoleftarrow$}}}

\usepackage{wrapfig}

\DeclareMathAlphabet{\mathpzc}{OT1}{pzc}{m}{it}

\newcommand{\Z}{\mathbb{Z}}

\newcommand{\R}{\mathbb{R}}

\newcommand{\C}{\mathbb{C}}

\newcommand{\cp}{\C \mathbf{P}}

\def\CA{\mathcal{A}}
\def\CB{\mathcal{B}}

\def\CE{\mathcal{E}}
\def\CF{\mathcal{F}}
\def\CH{\mathcal{H}}
\def\CN{\mathcal{N}}

\def\CM{\mathcal{M}}
\def\CT{\mathcal{T}}
\def\CW{\mathcal{W}}

\def\HVW{\mathcal{H}_{\mathrm{VW}}}
\def\ZVW{Z_{\mathrm{VW}}}
\def\MVW{\mathcal{M}_{\mathrm{VW}}}

\def\tilde{\widetilde}

\def\bar{\overline}

\def\Tr{\mathrm{Tr}\,}

\DeclareMathOperator{\Hom}{Hom}

\newcommand{\spinc}{\operatorname{Spin}^c}

\newcommand{\grdim}{\operatorname{gr-dim}}

\newcommand{\spin}{\mathfrak{s}}

\newcommand{\be}{\begin{equation}}
\newcommand{\bea}{\begin{eqnarray}}
\newcommand{\eea}{\end{eqnarray}} 
\newcommand{\ee}{\end{equation}}

\newtheorem{theorem}{Theorem}

\newtheorem{proposition}[theorem]{Proposition}
\newtheorem{lemma}[theorem]{Lemma}
\newtheorem{conjecture}[theorem]{Conjecture}

\newtheorem*{namedtheorem}{\theoremname}
\newcommand{\theoremname}{testing}

\theoremstyle{definition}
\newtheorem{definition}[theorem]{Definition}

\newtheorem{remark}[theorem]{Remark}

\title[3-manifolds and Vafa-Witten theory]{3-manifolds and Vafa-Witten theory}
\date{July 12, 2022}
 
\author[S. Gukov]{Sergei Gukov}
\address{California Institute of Technology, Pasadena, CA, 91125.\newline \it{E-mail address:} \tt{gukov@theory.caltech.edu}}
\thanks{S. G.~was supported by the National Science Foundation under Grant No.~NSF DMS 1664227 and by the U.S.~Department of Energy, Office of Science, Office of High Energy Physics, under Award No.~DE-SC0011632.}

\author[A. Sheshmani]{Artan Sheshmani}
\address{Harvard University, Jefferson Laboratory, 17 Oxford St, Cambridge, MA 02138.\newline \it{E-mail address:} \tt{artan@mit.edu}}
\thanks{Research of A. S. was partially supported by the NSF DMS-1607871, NSF DMS-1306313, the Simons 38558, and HSE University Basic Research Program. A.S. would like to further sincerely thank the Center for Mathematical Sciences and Applications at Harvard University, and Harvard University Physics department, IMSA University of Miami, as well as the Laboratory of Mirror Symmetry in Higher School of Economics, Russian federation, for the great help and support.}
  
\author[S.-T. Yau]{Shing-Tung Yau}
\address{Yau Mathematical Science Center, Tsinghua University, Beijing, 10084, China.\newline \it{E-mail address:} \tt{yau@math.harvard.edu}}
\thanks{The work of S.-T. Y. was partially supported by the Simons Foundation grant 38558.}

\numberwithin{equation}{section}

\begin{document}

\maketitle

\begin{abstract}
We initiate explicit computations of Vafa-Witten invariants of 3-manifolds, analogous to Floer groups in the context of Donaldson theory. In particular, we explicitly compute the Vafa-Witten invariants of 3-manifolds in a family of concrete examples relevant to various surgery operations (the Gluck twist, knot surgeries, log-transforms). We also describe the structural properties that are expected to hold for general 3-manifolds, including the modular group action, relation to Floer homology, infinite-dimensionality for an arbitrary 3-manifold, and the absence of instantons.
\end{abstract}

\tableofcontents

\section{Introduction}

The main goal of this paper is to compute and study invariants of 3-manifolds in Vafa-Witten theory \cite{Vafa:1994tf}, which is a particular generalization of the Donaldson gauge theory \cite{MR710056}. The latter involves the study of moduli spaces of solutions to the anti-self-duality equations
\be
F_A^+ \; = \; 0
\label{ASDeqs}
\ee
for the gauge connection $A$ over a 4-manifold $M_4$. When the 4-manifold is of the form (illustrated in Figure~\ref{fig:Hilbspace})
\be
M_4 \; = \; \R \times M_3
\label{M4RM3}
\ee
one can construct an infinite-dimensional version of the Morse theory on the space of gauge connections on $M_3$, called the instanton Floer homology \cite{MR956166}.
\parpic[r]{
	\begin{minipage}{50mm}
		\centering
		\includegraphics[scale=0.4]{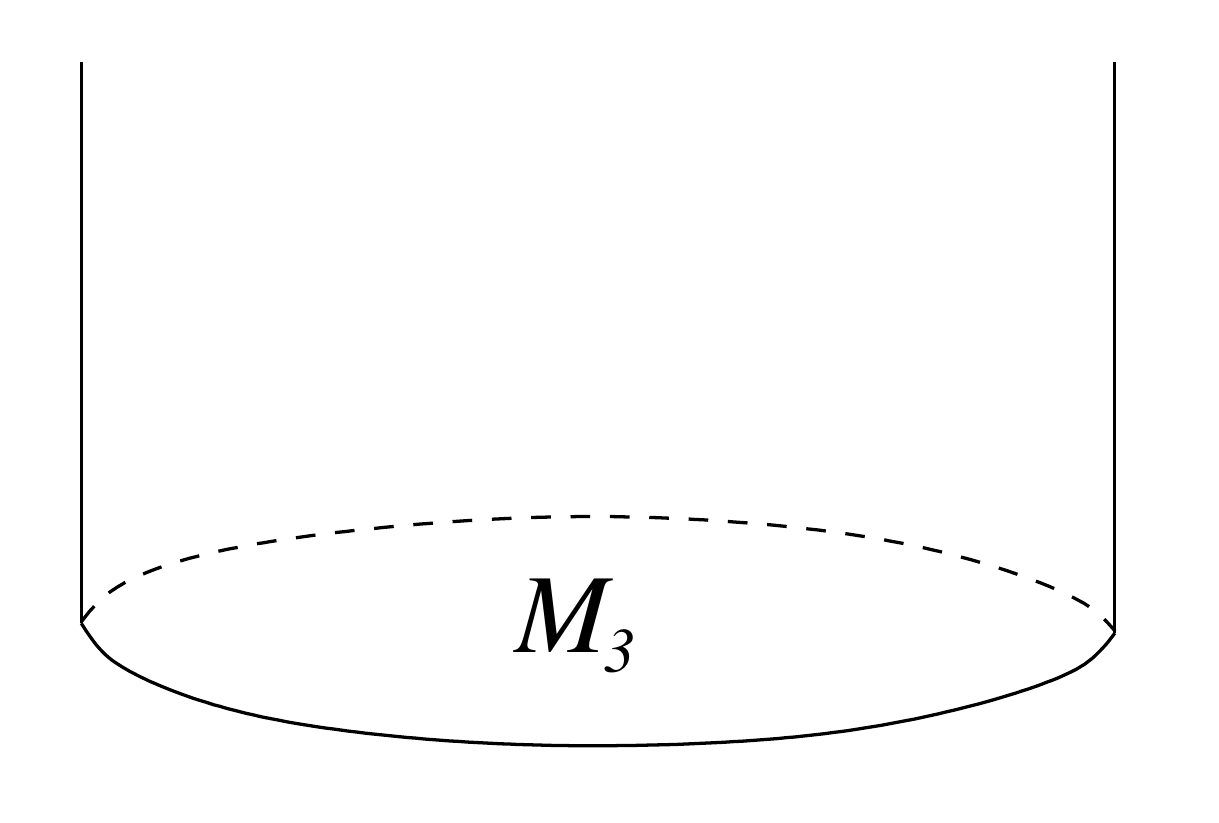}
		\captionof{figure}{The setup of a Floer theory. In physics, it represents the space of states of a 4-dimensional topological gauge theory on~$M_3$.}
		\label{fig:Hilbspace}
	\end{minipage}%
}
In particular, the Floer homology is a homology of a chain complex generated by $\R$-invariant (``stationary'') solutions to the PDEs \eqref{ASDeqs}, with the differential that comes from non-trivial $\R$-dependent ``instanton'' solutions on $\R \times M_3$. For introduction to Floer theory we recommend an excellent book \cite{MR1883043}. From the physics perspective, instanton Floer homology of a 3-manifold $M_3$ is the space of states in a Hamiltonian quantization of the topological Yang-Mills theory on $\R \times M_3$ \cite{Witten:1988ze}.

Since then, many variants of Floer homology have been studied, most notably the monopole Floer homology \cite{MR2388043} based on Seiberg-Witten monopole equations:
\begin{eqnarray}
F_A^+ & = & \Psi \otimes \bar \Psi - \frac{1}{2} (\bar \Psi \Psi){\rm Id}, \label{SWeq} \\
D \!\!\!\! \slash \, \Psi & = & 0 \nonumber
\end{eqnarray}
where, in addition to the $U(1)$ gauge connection $A$, the configuration space (the ``space of fields'') includes a section $\Psi \in \Gamma (M_4, W^+)$ of a complex spinor bundle $W^+$. The monopole Floer homology $HM (M_3, \spin)$ depends on a choice of additional data, namely the spin$^c$ structure $\spin \in \spinc (M_3)$, and is equivalent to the Heegaard Floer homology $HF (M_3, \spin)$ \cite{MR2102999} and to the embedded contact homology \cite{MR2827830,MR2532999}. In fact, technical details in each of these theories lead to four different variants, which correspondingly match.
\parpic[l]{
	\begin{minipage}{50mm}
		\centering
		\includegraphics[scale=0.4]{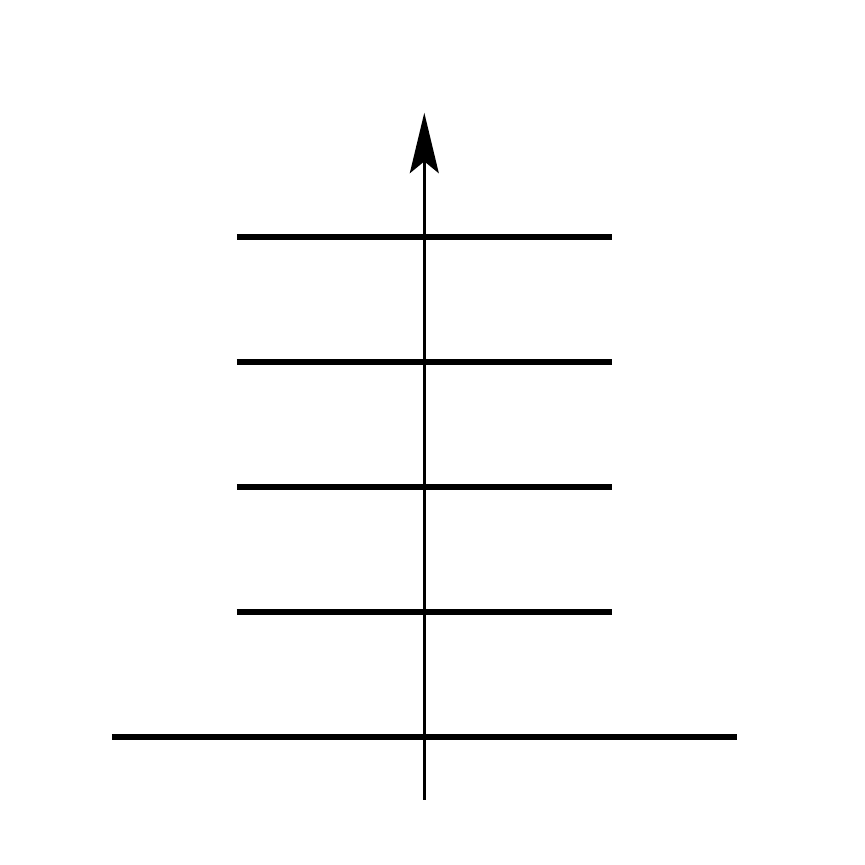}
		\captionof{figure}{The space $\CT^+$ is isomorphic to the Hilbert space of a quantum harmonic oscillator.}
		\label{fig:spectrum}
	\end{minipage}%
}
The variant that will be most relevant to us in what follows is the so-called ``to'' version of the monopole Floer homology, $\widecheck{HM} (M_3, \spin)$, and the corresponding ``plus'' version of the Heegaard Floer homology, denoted $HF^+ (M_3, \spin)$. The equivalence between these Floer theories will be useful to us because the monopole Floer homology will be conceptually closer to its analogue in Vafa-Witten theory, whereas concrete calculations are usually simpler in the Heegaard Floer homology. In particular, before we turn to PDEs in Vafa-Witten theory, let us briefly mention a few concrete results in the Heegaard Floer homology which, on the one hand, will serve as a prototype in our study of Vafa-Witten invariants of 3-manifolds and, on the other hand, illustrate the general structure of Floer homology in Yang-Mills theory \eqref{ASDeqs} and in Seiberg-Witten theory \eqref{SWeq}.

\begin{theorem}[based on {\cite[Theorem 9.3]{MR2065507}}]\label{thm:HFbasic}
Let $L(p,1)$ denote a Lens space and $\Sigma_g$ be a closed oriented surface of genus $g$. Then,
\begin{subequations}\label{HFbasic}
\begin{eqnarray}
HF^+ (L(p,1), \spin) & = & \CT^+_0, \qquad\qquad \forall \spin \\
HF^+ (S^2 \times S^1, \spin) & = & \CT^+_{-1/2} \oplus \CT^+_{1/2}, \qquad \spin = \spin_0 \\
HF^+ (\Sigma_g \times S^1, \spin_h) & = & \bigoplus_{i=0}^d \Lambda^i H^1 (\Sigma_g ; \Z) \otimes \CT^+_0 / (U^{i-d-1}), \qquad h \ne 0
\end{eqnarray}
\end{subequations}
where $d = g-1-|h|$ and $\spin_h$ is the spin$^c$ structure with $c_1 (\spin_h) = 2h [S^1]$.
\end{theorem}
This list of basic but important calculations in the Heegaard Floer homology illustrates well a key ingredient that plays a central in this theory, namely the space
\be
\CT^+ = \C [U, U^{-1}] / U \cdot \C [U] \cong H^*_{U(1)} (\text{pt}) = H^* (\cp^{\infty}) \cong \C [u]
\label{Ttower}
\ee
From the physics perspective \cite{Gukov:2016gkn}, this space can be understood as the Fock space of a single boson, $\CT^+ \cong \text{Sym}^* (\phi)$, illustrated in Figure~\ref{fig:spectrum}. When the minimal degree (of the ``ground state'') is equal to $n$ we write $\CT^+_n$. It is often convenient to work with the corresponding Poincar\'e polynomial, $\frac{t^n}{1 - t^2}$, where we introduced a new variable $t$ and took into consideration that in standard conventions $U^{-1}$ carries homological degree $2$. The same variable $t$, with the same meaning, will be used in the context of Vafa-Witten theory, to which we turn next.

Our main goal is to set up a concrete framework that allows to compute analogous invariants of 3-manifolds in Vafa-Witten theory, where the relevant PDEs generalize \eqref{ASDeqs} and \eqref{SWeq}:
\be
\begin{aligned}
	F_A^+ - \frac{1}{2} [B \times B] + [C,B] & = 0 \\
	d_A^* B - d_A C & = 0
\end{aligned}
\qquad \text{where} \qquad
\begin{aligned}
	A & \in \CA_P \\
	B & \in \Omega^{2,+} (M_4; \text{ad}_P) \\
	C & \in \Omega^{0} (M_4; \text{ad}_P)
\end{aligned}
\label{VWeqs}
\ee
As explained in the main text, these equations have a number of parallels and relations to \eqref{ASDeqs} and \eqref{SWeq}. This will be the basis for various structural properties of the Floer homology groups in Vafa-Witten theory which, following \cite{Gukov:2016gkn}, we denote by $\HVW (M_3)$, {\it cf.} \eqref{HVWviaQ}. In particular, reflecting the fact that the configuration space in \eqref{VWeqs} is much larger than in \eqref{ASDeqs} and \eqref{SWeq}, we will see that $\HVW (M_3)$ is also much larger, in particular, compared to $HF^+ (M_3)$. Also, many challenges that one encounters in constructing 3-manifold invariants based on \eqref{ASDeqs} and \eqref{SWeq} will show up in the Vafa-Witten theory as well.

Our main result is a concrete framework that allows computation of $\HVW (M_3)$ for many simple 3-manifolds. In particular, we produce a suitable analogue of Theorem~\ref{thm:HFbasic} in Vafa-Witten theory which, due to a large size of $\HVW (M_3)$, we state here only at the level of the Poincar\'e series, relegating the full description of $\HVW (M_3)$ to the main text.
\begin{proposition}\label{thm:VWbasic}
For $G=SU(2)$:
\begin{subequations}\label{VWbasic}
\begin{eqnarray}
\grdim \HVW (S^3) & = & \frac{1}{1-t^2} \\
\grdim \HVW (S^2 \times S^1) & = & \frac{2 t^{3/2} \left(t x^4+1\right)}{\left(1- t^2\right) \left(1-t^2 x^4\right)} \\
\grdim \HVW (\Sigma_g \times S^1) & = &
\sum_{\lambda=0}^9 S_{0 \lambda}^{2-2g}
\end{eqnarray}
\end{subequations}
where $t$ has the same meaning as in \eqref{HFbasic} and the explicit values of $S_{0 \lambda}$ are summarized in \eqref{Selements}.
\end{proposition}
We present two approaches to these results, as part of a more general framework for computing $\HVW (M_3)$. First, in section~\ref{sec:eqVerlinde} we carefully analyze gradings on $\HVW (M_3)$ for general 3-manifolds as well as for circle bundles over $\Sigma_g$. In the latter case, the relevant moduli spaces turn out to be closely related to the moduli spaces of complex $G_{\C}$ connections on $\Sigma_g$, whose non-compactness can be compensated by additional symmetries, as in the equivariant Verlinde formula \cite{MR3670728}. Then, in section~\ref{sec:Qcohomology} we reproduce the same results by a direct computation of $Q$-cohomology (a.k.a. the BRST cohomology) in the Vafa-Witten theory.

Apart from analyzing the main statement of Proposition~\ref{thm:VWbasic} via several methods, in this paper we also present evidence for a number of conjectures --- Conjectures~\ref{conj:infdimnl}, \ref{conj:noinstantons}, and \ref{conj:RGflows} --- all of which are concrete falsifiable statements. Part of our motivation is that future efforts to either prove or disprove these conjectures can lead to better understanding of the Vafa-Witten theory on 3-manifolds.

Another motivation for this work is to develop surgery formulae in Vafa-Witten theory, analogous to those in Yang-Mills theory \cite{MR1171893,MR1362829} and in Seiberg-Witten theory \cite{MR1492130,MR1650308,MR2299739}. Initial steps in this direction were made in \cite{MR4047476} where simple instances of cutting and gluing along 3-manifolds were considered in the context of Vafa-Witten theory. One of our goals here is to generalize these recent developments and to bring them closer to the above-mentioned surgery formulae by studying Vafa-Witten invariants of 3-manifolds. In the long run, this could be a strategy for computing Vafa-Witten invariants on general 4-manifolds, beyond the well studied class of K\"ahler surfaces.

Finally, aiming to make this paper accessible to both communities, we tried to not overload it with mathematics or physics jargon. Hopefully, we managed to strike the right balance.

\subsection*{Acknowledgements}
The authors would like to thank Tobias Ekholm, Po-Shen Hsin, Martijn Kool, Nikita Nekrasov, Du Pei, and Yuuji Tanaka for useful discussions during recent years related to various aspects of this work.

\section{Predictions from MTC$[M_3]$ and the equivariant Verlinde formula}
\label{sec:eqVerlinde}

We start by summarizing the symmetries of Vafa-Witten theory and demonstrating how these symmetries can help to deal with non-compact moduli spaces.

\subsection{Symmetries and gradings}

The holonomy group of a general 4-manifold $M_4$ is
\be
SO(4)_E \; \cong \; SU(2)_{\ell} \times SU(2)_r
\ee
where we use the subscript ``$E$'' to distinguish it from other symmetries, which will enter the stage shortly. When $M_4$ is of the form \eqref{M4RM3}, the holonomy is reduced to
\be
SU(2)_E \; = \; \text{diag} [ SU(2)_{\ell} \times SU(2)_r ]
\label{3dholonomy}
\ee
And, when $M_3 = \Sigma_g \times S^1$, it is further reduced to $SO(2)_E \cong U(1)_E \subset SU(2)_E$.

In addition, 4d $\CN=4$ super-Yang-Mills theory has ``internal'' $R$-symmetry $SO(6)_R$. In the process of a topological twist, required to define the theory on a general 4-manifold \cite{Vafa:1994tf}, this symmetry is broken to a subgroup $SU(2)_R \subset SO(6)_R$. All fields and states in the topological theory form representations under this group. Its Cartan subgroup, which we denote by $U(1)_t \subset SU(2)_R$, is familiar from the Donaldson-Floer theory and also plays an important role in this paper. In particular, it provides a grading to the Floer homology groups, which are $\Z$-graded in the Vafa-Witten theory.\footnote{In the Donaldson-Floer theory, it is reduced to a $\Z / 8\Z$ grading, which is a manifestation of an anomaly in the physical super-Yang-Mills theory. There is no such anomaly in the Vafa-Witten theory.} We use the variable $t$ to write the corresponding generating series of their graded dimensions:
\begin{definition}
\be
\grdim \HVW (M_3) \; := \; \sum_n t^n \dim \HVW^n (M_3)
\label{grdimdef}
\ee
\end{definition}

When 4-manifold $M_4$ is not generic, {\it i.e.} has reduced holonomy, a smaller subgroup of the original $R$-symmetry $SO(6)_R$ is used in the topological twist and, as a result, a larger part of this symmetry may remain unbroken. Specifically, on a 3-manifold Vafa-Witten theory has $SU(2)_R \times \overline{SU(2)}_R$ symmetry, and on a 2-manifold this internal symmetry is enhanced further to $Spin(4)_R \times U(1)_R$.

All these cases are summarized in Table~\ref{tab:symmetries} where, following the notations of \cite{Dijkgraaf:1996tz,Blau:1996bx}, we also list the number of unbroken supercharges, $N_T$.
These numbers are easy to see by noting that the supercharges in Vafa-Witten theory transform as $({\bf 2}, {\bf 2}, {\bf 2}) \oplus ({\bf 1}, {\bf 3}, {\bf 2}) \oplus ({\bf 1}, {\bf 1}, {\bf 2})$ under $SU(2)_{\ell} \times SU(2)_r \times SU(2)_R$ \cite{Vafa:1994tf}. Then, using the second column in Table~\ref{tab:symmetries} and counting the number of singlets under the holonomy group gives $N_T$ in each case.

\begin{table}[ht]
\begin{centering}
\begin{tabular}{|c|c|c|c|}
\hline
~$\phantom{\int^{\int^\int}} ~M_4~ \phantom{\int_{\int}}$~ & ~~holonomy~~ & ~~$R$-symmetry~~ & ~~~$N_T$~~~
\tabularnewline
\hline
\hline
$\phantom{\int^{\int^\int}} \text{generic} \phantom{\int_{\int}}$ & $SU(2)_{\ell} \times SU(2)_r$ & $SU(2)_R$ & $2$
\tabularnewline
\hline
$\phantom{\int^{\int^\int}} \R \times M_3 \phantom{\int_{\int}}$ & $\text{diag} [ SU(2)_{\ell} \times SU(2)_r ]$ & $SU(2)_R \times \overline{SU(2)}_R$ & $4$
\tabularnewline
\hline
$\phantom{\int^{\int^\int}} \R \times S^1 \times \Sigma_g \phantom{\int_{\int}}$ & $SO(2)_E$ & $Spin(4)_R \times U(1)_R$ & $8$
\tabularnewline
\hline
\end{tabular}
\par\end{centering}
\caption{\label{tab:symmetries} Symmetries of the Vafa-Witten theory on different manifolds.}
\end{table}

The symmetry $SU(2)_{\ell} \times SU(2)_r \times SU(2)_R$ of the Vafa-Witten theory is also useful for keeping track of various fields that appear in PDEs and lead to moduli spaces of solutions. Apart from the gauge connection, all ordinary bosonic (as opposed to Grassmann odd) variables originate from scalar fields of the 4d $\CN=4$ super-Yang-Mills theory. After the topological twist they transform as
\be
\underbrace{({\bf 2}, {\bf 2}, {\bf 1})}_A \oplus \underbrace{({\bf 1}, {\bf 3}, {\bf 1})}_B \oplus \underbrace{({\bf 1}, {\bf 1}, {\bf 3})}_{C, \phi, \bar \phi}
\label{VWbosfields}
\ee
under $SU(2)_{\ell} \times SU(2)_r \times SU(2)_R$. In particular, we see that, apart from the fields $A$, $B$, and $C$ that already appeared in the PDEs \eqref{VWeqs}, the full theory also contains fields $\phi$ and $\bar \phi$ that sometimes vanish and, therefore, can be ignored on closed 4-manifolds with a special metric, but will play an important role below, in the computation of the Floer homology groups. Using \eqref{3dholonomy} we also see that, upon reduction to three dimensions, the Vafa-Witten theory contains a gauge connection and fields that transform as $({\bf 1}, {\bf 1}) \oplus ({\bf 3}, {\bf 1}) \oplus ({\bf 1}, {\bf 3})$ under $SU(2)_E \times SU(2)_R$, reproducing one of the results in \cite{Blau:1996bx}. In other words, the theory contains two 1-forms, which naturally combine into a complexified gauge connection, so that the space of solutions on a 3-manifold contains the space of complex flat connections, $\CM_{\text{flat}} (M_3, G_{\C})$.
This also holds true for another twist of 4d $\CN=4$ super-Yang-Mills \cite{Marcus:1995mq}, which under reduction to three dimensions gives the same theory.

If we continue this process further, and reduce the Vafa-Witten theory to a two-dimensional theory on $\Sigma$, from the above it follows that the resulting theory contains two 1-form fields and four complex 0-form Higgs fields, all in the adjoint representation of the gauge group $G$. These fields comprise a $\tilde \CE$-valued $G$-Higgs bundles on $\Sigma$,
\be
\tilde \CE \; = \; L_1 \oplus L_2 \oplus L_3 \oplus L_4
\,, \qquad L_i = K^{R_i/2} \qquad (R_i \in \Z)
\label{ESigma}
\ee
with $R_i = (2,0,0,0)$. As in the classical work of Hitchin \cite{MR887284}, one of the Higgs fields here (say, the one associated with $L_4$) comes from dimensional reduction of 4d gauge connection to two dimensions. Therefore, the Vafa-Witten theory on $M_3 = S^1 \times \Sigma$ can be viewed as a four-dimensional lift of the $\tilde \CE$-valued $G$-Higgs bundles on $\Sigma$, without the last term in \eqref{ESigma}:
\be
\CE \; = \; L_1 \oplus L_2 \oplus L_3
\,, \qquad L_i = K^{R_i/2} \qquad (R_i \in \Z)
\label{EVW}
\ee
where $R_i = (2,0,0)$. Put differently, on $M_3 = S^1 \times \Sigma$ the fields \eqref{VWbosfields} in four-dimensional Vafa-Witten theory consist of a gauge connection and three copies of the Higgs field. Below, we shall refer to this collection of fields as the $\CE$-valued $G$-Higgs bundles on $\Sigma$ and will be interested in the computation of its K-theoretic (three-dimensional) and elliptic (four-dimensional) equivariant character.\footnote{Note, all fields in the Vafa-Witten theory are valued in the adjoint representation of the gauge group $G$, {\it cf.} \eqref{VWeqs}. And, the terminology is such that ``$\CE$-valued $G$-Higgs bundles on $\Sigma$'' in the present context describe a version of the Hitchin equations on $\Sigma$, with Higgs fields valued in $L_1 \otimes \text{ad}_P \oplus L_2 \otimes \text{ad}_P \oplus L_3 \otimes \text{ad}_P$.}

\begin{remark} \label{rem:trace} By definition, the generating series of graded dimensions \eqref{grdimdef} is the Vafa-Witten invariant on $M_4 = S^1 \times M_3$ with a holonomy for $U(1)_t$ symmetry along the $S^1$. In other words, the variable $t$ is the holonomy of a background $U(1)_t$ connection on the $S^1$. (Hopefully, this also clarifies our choice of notations.) Note, in a $d$-dimensional TQFT that satisfies Atiyah's axioms, the relation $Z (S^1 \times M_{d-1}) = \dim \CH (M_{d-1})$ is simply one of the axioms. Although the Vafa-Witten theory is not a TQFT in this sense --- in part, because $\HVW (M_3)$ is infinite-dimensional for any $M_3$ --- a version of the relation $Z (S^1 \times M_3) = \dim \CH (M_3)$ still holds if instead of the ordinary dimension we use the graded dimension \eqref{grdimdef}. And, it is also useful to compare this modification of the Vafa-Witten invariant of $S^1 \times M_3$, with the holonomy $t$, to the ordinary Vafa-Witten invariant of $M_4 = S^1 \times M_3$. Since this 4-manifold is K\"ahler for many $M_3$, we can use the calculation of Vafa-Witten invariants on K\"ahler surfces \cite{Vafa:1994tf,Dijkgraaf:1997ce} (see also \cite{MR4158461,MR4093873,MR4084159,Manschot:2021qqe} for recent work and mathematical proofs):
\begin{multline}
\ZVW (M_4) =
\sum_{\genfrac{}{}{0pt}{2}{x: \mathrm{basic}}{\mathrm{classes}}}
\mathrm{SW}_{M_4} (x) \Big[
(-1)^{\frac{\chi+\sigma}{4}} \delta_{v,[x']}
\left( \frac{G(q^2)}{4} \right)^{\frac{\chi+\sigma}{8}}
\left( \frac{\theta_0}{\eta^2} \right)^{-2\chi - 3 \sigma}
\left( \frac{\theta_1}{\theta_0} \right)^{-x' \cdot x'}
\\
+  2^{1 - b_1} (-1)^{[x'] \cdot v}
\left( \frac{G(q^{1/2})}{4} \right)^{\frac{\chi+\sigma}{8}}
\left( \frac{\theta_0 + \theta_1}{2\eta^2} \right)^{-2\chi - 3 \sigma}
\left( \frac{\theta_0 - \theta_1}{\theta_0 + \theta_1} \right)^{-x' \cdot x'}
\\
+  2^{1 - b_1} i^{-v^2} (-1)^{[x'] \cdot v}
\left( \frac{G(-q^{1/2})}{4} \right)^{\frac{\chi+\sigma}{8}}
\left( \frac{\theta_0 - i \theta_1}{2\eta^2} \right)^{-2\chi - 3 \sigma}
\left( \frac{\theta_0 + i \theta_1}{\theta_0 - i \theta_1} \right)^{-x' \cdot x'} \Big]
\label{VWKahler}
\end{multline}
which allows to express the result in terms of the Seiberg-Witten invariants $\mathrm{SW}_{M_4}$ and basic topological invariants, such as the Euler characteristic $\chi$ and the signature $\sigma$. Because $\chi = 0 = \sigma$ for $M_4 = S^1 \times M_3$, we quickly conclude that the ordinary Vafa-Witten invariant (without holonomy along the $S^1$) is simply a number, independent of $q$ or other variables. The calculation of \eqref{VWbasic} is the analogue of \eqref{VWKahler} in the presence of holonomies along the $S^1$. In particular, much like \eqref{VWKahler}, it has no $q$-dependence and enjoys a relation to the Seiberg-Witten theory (more precisely, $HF^+ (M_3)$, {\it cf.} \eqref{HFbasic}) as we explain below.
\end{remark}

\begin{remark} \label{rem:unrefined}
On K\"ahler manifolds, the derivatiion of \eqref{VWKahler} deals with one of the main challenges in the Vafa-Witten theory: the non-compactness of moduli spaces of solutions to PDEs. This issue has many important ramifications and consequences, {\it e.g.} it prevents the theory to be a TQFT in the traditional sense. In this section, this issue is addressed with the help of symmetries and the corresponding equivariant parameters, reducing the calculations to compact fixed point sets. Therefore, comparing the results of such calculations, {\it e.g.} \eqref{VWbasic}, to the partition functions on $M_4 = S^1 \times M_3$ without holonomies along the $S^1$, such as \eqref{VWKahler}, can not be achieved simply by ``tunring off'' the holonomies, {\it i.e.} taking the limit $x \to 1$ and $t \to 1$. In addition, one needs to regularize in some way the contribution of zero-modes (non-compact directions) that make \eqref{VWbasic} singular in this limit. The difficulty is that, in a non-abelian theory on a general manifold, there is no canonical way to do this because all fields interact with one another. In the case of \eqref{VWbasic}, one could simply multiply this expression by a suitable factor before taking the limit; we will determine the precise factor in Lemma~\ref{lemma:xtasymptotics} and then present a thorough discussion of the corresponding zero-modes in section~\ref{sec:Qcohomology}.
\end{remark}

\subsection{Cohomology of $\CE$-valued Higgs bundles}

The ring of functions on $\CM_{\mathrm{flat}} (M_3, G_{\C})$ and, more generally, cohomology of $\CE$-valued Higgs bundles naturally appear in a slightly different but related problem. Here, we briefly outline the connection, in particular because it suggests that we should expect particular structural properties of $\HVW (M_3)$, such as the action of the modular group $SL(2, \Z)$ and the mapping class group of $M_3$:
\be
\mathrm{MCG} (M_3) \times SL(2, \Z) \; \acts \; \HVW (M_3)
\label{MCGaction}
\ee
as well as relations to the familiar variants of the Floer homology.

Starting in six dimensions and reducing on $T^2 \times M_3$, with a partial topological twist along $M_3$, we obtain the space of states in quantum mechanics that can be viewed in several equivalent ways \cite{Gukov:2016gkn}. Reducing on $T^2$ first (and also taking the limit $\mathrm{vol} (T^2) \to 0$) we obtain the Vafa-Witten theory on $\R \times M_3$ that leads to $\HVW (M_3)$. On the other hand, first reducing on $M_3$ gives a 3d theory $T[M_3]$. Its space of supersymmetric states on $T^2$ can be also viewed as the space of states in 2d A-model $T^{\text{A}}[M_3]$, whose target space is essentially $\CM_{\mathrm{flat}} (M_3, G_{\C})$. In this way we obtain a chain of approximate relations
\be
\HVW (M_3) \; \simeq \; \CH_{T[M_3]} (T^2) \; \simeq \; \CH_{T^{\text{A}}[M_3]} (S^1) \; \cong \; QH^* (\CM_{\text{flat}} (M_3, G_{\C}))
\label{HHHQH}
\ee
where {\it e.g.} in the last relation we only wrote the most interesting part of the moduli space\footnote{This point will be properly addressed in the rest of the paper, where all fields and all moduli will be taken into account. In particular, when $M_3 = S^1 \times \Sigma$, incorporating the extra moduli gives precisely the moduli space of $\CE$-valued Higgs bundles, and quantum cohomology in \eqref{HHHQH} is replaced by the classical cohomology.} and essentially identified $T[M_3]$ with $T^{\text{A}}[M_3]$. These two theories are related by a circle compactification, and if the circle has finite size the quantum cohomology in \eqref{HHHQH} should be replaced by the quantum (equivariant) K-theory $QK (\CM_{\text{flat}} (M_3, G_{\C}))$. A similar approximation enters the first relation in \eqref{HHHQH} because $\HVW (M_3)$ and $\CH_{T[M_3]} (T^2)$ differ by the Kaluza-Klein modes on $T^2$. This somewhat delicate point is often overlooked when 6d theory on a finite-size torus is treated as 4d $\CN=4$ super-Yang-Mills. Luckily, the role of these KK modes, which we plan to address more fully elsewhere, is not very important for the aspects of our interest here,\footnote{For example, in the computation of the equivariant character of the moduli space of $\CE$-valued Higgs bundles, taking the limit $\mathrm{vol} (S^1) \to 0$, {\it i.e.} replacing $T^2$ by $S^1$ in the computation of \eqref{Sxyt}--\eqref{Selements} has the effect of reducing the set $\{ \lambda \}$, so that instead of 10 possible values it contains only 5. However, since the remaining $\lambda$'s have the same values of $S_{0 \lambda}$ as the eliminated ones, this changes the calculation of \eqref{ZviaSmatrix} only by an overall factor of 2. This is a general property of any theory with matter fields in the adjoint representation of the gauge group $G=SU(2)$, which is certainly the case for the Vafa-Witten theory we are interested in.} in particular, it does not affect the symmetry \eqref{MCGaction}.

It was further proposed in \cite{Gukov:2016gkn} that \eqref{HHHQH} can be categorified into a higher algebraic structure, dubbed MTC$[M_3]$, that also controls the BPS line operators and twisted indices in 3d theory $T[M_3]$. In other words, \eqref{HHHQH} should arise as the Grothendieck ring of the category MTC$[M_3]$, which in general may not be unitary or semi-simple:
\be
K^0 (\text{MTC} [M_3]^{ss} ) \; \cong \; QK (\CM_{\text{flat}} (M_3, G_{\C}))
\label{KMTCss}
\ee
When $G_{\C}$ flat connections on $M_3$ are isolated, they are expected to correspond to simple objects in MTC$[M_3]$. Moreover, its Grothendieck group is expected to enjoy the action of the modular group $SL(2,\Z) = \mathrm{MCG} (T^2)$, which comes from the symmetry of $T^2$ and in the Vafa-Witten theory can be thought of the $S$-duality.
Similarly, the other part of the symmetry in \eqref{MCGaction}, namely $\mathrm{MCG} (M_3)$ comes from the $M_3$ part of the background and can be thought of as the duality of the 3d theory $T[M_3]$.

Although connections between different theories outlined here will not be used in the rest of the paper, they certainly help to understand the big picture and to see the origin of various structural properties of $\HVW (M_3)$. This includes the relation to the generalized cohomology of the moduli spaces and the action of the modular group in \eqref{MCGaction}. Curiously, it also suggests a relation to the Heegaard Floer homology of $M_3$. Namely, for $G=U(1)$ the space \eqref{HHHQH} can be identified with $\widehat{HF} (M_3) \otimes \C$, which Kronheimer and Mrowka \cite{MR2652464} conjectured to be isomorphic to the framed instanton homology, $I^{\#} (M_3) := I (M_3 \# T^3)$:
\be
I^{\#} (M_3) \cong \widehat{HF} (M_3; \C)
\label{framedinstanton}
\ee
On the other hand, because the space \eqref{HHHQH} with $G=U(1)$ is basically ``the Cartan part'' of $\HVW (M_3)$ for $G=SU(2)$, it suggests that the latter should contain the (framed) instanton Floer homology. Notice, while this argument used relations between various theories, the conclusion can be phrased entirely in the context of Vafa-Witten theory. As we shall see later, via direct analysis of the $Q$-cohomology and spectral sequences in the Vafa-Witten theory, this conclusion is on the right track (see {\it e.g.} \eqref{HFinHVW} and discussion that follows).

\parpic[r]{
	\begin{minipage}{40mm}
		\centering
		\includegraphics[scale=0.4]{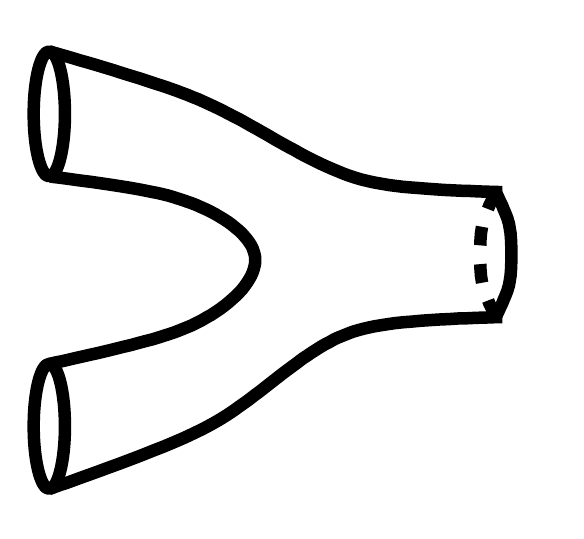}
		\captionof{figure}{Basic 2d cobordism that defines a product.}
		\label{fig:pants}
	\end{minipage}%
}
\subsection{Derivation of (\ref{VWbasic}b) and (\ref{VWbasic}c)}

As explained above, for $M_3 = S^1 \times \Sigma$ the graded dimension \eqref{grdimdef} is equal to the equivariant Verlinde formula for $\CE$-valued $G$-Higgs bundles on $\Sigma$, with $\CE$ as in \eqref{EVW}. In general, for $K^{R/2}_{\Sigma}$-valued Higgs fields with any $R$, it has the form of the ordinary Verlinde formula or A-model partition function on $\Sigma$ \cite{MR3670728} (see also \cite{Andersen:2016hoj,Halpern-Leistner:2016uay} for further mathematical developments). In other words, it is described by a 2d semisimple TQFT with a finite-dimensional space of states on $S^1$ (despite the fact that the space of states in Vafa-Witten theory is infinite-dimensional). Let $\{  \lambda \}$ be a basis of states that diagonalizes multiplication of the corresponding Frobenius algebra, called {\it the equivariant Verlinde algebra}, associated to a pair-of-pants and illustrated in Figure~\ref{fig:pants}.
We denote the corresponding eigenvalues of the structure constants by $(S_{0 \lambda})^{-1}$; their values will be determined shortly. Then,
\be
\grdim \HVW (M_3) \; = \; \sum_{\lambda} S_{0 \lambda}^{2-2g}
\label{ZviaSmatrix}
\ee
The sum runs over the set of admissible solutions to the Bethe ansatz equation for $\mathbb{T}$-valued variable $z$, {\it i.e.} the set of solutions from which those fixed by the Weyl group of $G$ are removed.

Just like in Donaldson theory one always starts with the gauge group $G=SU(2)$, in this paper we mainly focus on Vafa-Witten theory with $G=SU(2)$. Then, $\mathbb{T} = U(1)$ and we need to solve one Bethe equation for a single variable $z$:
\be
1 \; = \; \exp \left( \frac{\partial \tilde \CW}{\partial \log z} \right)
\label{BAE}
\ee
The values of $(S_{0\lambda})^{-2} = e^{2\pi i \Omega} \frac{\partial^2 \tilde \CW}{(\partial \log z)^2} \big\vert_{z_\lambda}$ consist of two factors, each evaluated on the solutions to \eqref{BAE}. One factor is simply the second derivative of the same function $\tilde \CW (z)$, called the {\it twisted superpotential}, that determines the Bethe equation itself. Both the Bethe equation and the other factor $e^{2\pi i \Omega}$, sometimes called the {\it effective dilaton}, are multiplicative in charged matter fields, in fact, in weight spaces of $\mathfrak{g} = \mathrm{Lie} (G)$, whereas $\tilde \CW (z)$ is additive. Specifically, a $K^{R/2}_{\Sigma}$-valued Higgs field contributes to the Bethe ansatz equation a factor
\be
L = K^{R/2}_{\Sigma}: \qquad
\exp \left( \frac{\partial \tilde \CW}{\partial \log z} \right) \; = \;
\frac{(t-z^2)^2}{(t z^2-1)^2}
\label{BetheHiggs}
\ee
where, as before, $z$ is the equivariant parameter for the gauge symmetry, while $t$ is the analogous equivariant parameter for a $U(1)$ (or $\C^*$) symmetry acting on the adjoint Higgs field by phase rotation (and dilation).
In particular, it follows that the {\it additive} contribution of a $K^{R/2}_{\Sigma}$-valued Higgs field to $\frac{\partial^2 \tilde \CW}{(\partial \log z)^2}$ is $\frac{4}{z^2 t^{-1} - 1} - \frac{4}{z^2 t - 1}$.
Similarly, a $K^{R/2}_{\Sigma}$-valued Higgs field contributes to $e^{2\pi i \Omega}$ a factor
\be
e^{2\pi i \Omega_{\mathrm{Higgs}}} \; = \;
\left( \frac{t^{3/2} z^2}{(t-1) (t-z^2) (t z^2-1)} \right)^{R-1}
\ee
Note that Higgs fields with $R=0$ and $R=2$ produce opposite contributions; this feature will play a role below and can be seen directly in the calculations of one-loop determinants \cite{MR3670728} that lead to the expressions quoted here.
While the $SU(2)$ gauge field (or, rather, superfield) does not contribute directly to the Bethe ansatz equation \eqref{BAE}, it contributes to $S_{0\lambda}^2$ a factor
\be
e^{- 2\pi i \Omega_{\mathrm{gauge}}} \; = \; \frac{1}{z^2} - 2 + z^2
\ee

Now we are ready to put these ingredients together and compute \eqref{ZviaSmatrix} for $\CE$-valued Higgs bundles on $\Sigma$, with $\CE$ in the form \eqref{EVW}. In particular, corresponding to the three terms in \eqref{EVW}, there are three factors \eqref{BetheHiggs} in the Bethe ansatz equation,
\be
1 \; = \; \exp \left( \frac{\partial \tilde \CW}{\partial \log z} \right) \; = \;
\frac{\left(t-z^2\right)^2 \left(x-z^2\right)^2 \left(y-z^2\right)^2}{\left(t z^2-1\right)^2 \left(x z^2-1\right)^2 \left(y z^2-1\right)^2}
\label{Betheeq}
\ee
where, in addition to $z$, we introduced three equivariant parameters $(x,y,t)$ associated with each of the terms in \eqref{EVW}. Equivalently, these are the equivariant parameters for the symmetry $U(1)_x \times U(1)_y \times U(1)_t$ of $\CE$-valued Higgs bundles on $\Sigma$; using Table~\ref{tab:symmetries} and the discussion around it, this symmetry can be identified with the maximal torus of the group $Spin(4)_R \times U(1)_R$ in the last row. From that discussion we also know that only $U(1)_t$ subgroup of $U(1)_x \times U(1)_y \times U(1)_t$ admits a lift to the Vafa-Witten on a general 4-manifold. In other words we can use $U(1)_x \times U(1)_y \times U(1)_t$ and the corresponding equivariant parameters for the equivariant Verlinde formula in the case of $M_3 = S^1 \times \Sigma$, but need to set $x=1$ and $y=1$ when we work with more general 3-manifolds and 4-manifolds.

There are 10 admissible solutions to \eqref{Betheeq}, {\it i.e.} 10 values $z_{\lambda}$ not fixed by the Weyl group of $G=SU(2)$. Aside from a simple pair of solutions $z = \pm i$ that can be seen with a naked eye, the expressions for $z_{\lambda}$ as functions of $(x,y,t)$ are not very illuminating. In fact, to simplify things further, we often find it convenient to set $x=y$. (Recall, that on general manifolds one needs to set $x=y=1$.)
Indeed, we should expect a simplication in this limit because the values of $R_i$ corresponding to $(x,y,t)$ are $(R_1, R_2, R_3) = (2,0,0)$, and the contributions to $e^{2\pi i \Omega}$ with $R=0$ and $R=2$ cancel each other, as was noted above.

Combining all contributions described above, a straightforward but slightly tedious calculation gives
$$
S_{00}^{2} \; = \; S_{01}^{2} \; = \;
\frac{t^{3/2} (x-1) (x+1)^3 y^{3/2}}{\left(t^2-1\right) x^{3/2} \left(y^2-1\right) (t (3 x y+x+y-1)+x (y-1)-y-3)}
$$
\begin{multline}
S_{02}^{2} \; = \; S_{03}^{2} \; = \; S_{04}^{2} \; = \; S_{05}^{2} \; = \; 
\\
\; = \;
\frac{t^{3/2} (x-1)^3 y^{3/2} (t x-1) (x y-1)}{4 (t-1) x^{3/2} (y-1) (t y-1) (t x y-1) (t (3 x y+x+y-1)+x (y-1)-y-3)}
\label{Sxyt}
\end{multline}
$$
S_{06}^{2} \; = \; S_{07}^{2} \; = \; S_{08}^{2} \; = \; S_{09}^{2} \; = \;
\frac{t^{3/2} \left(x^2-1\right) y^{3/2} (t x-1) (x y-1)}{4 \left(t^2-1\right) x^{3/2} \left(y^2-1\right) (t y-1) (t x y+1)}
$$
for generic $(x,y,t)$. Specializing to $x=y$ we obtain much simpler and easier to read expressions:
\be
\begin{array}{l}
S_{00}^{2} \; = \; S_{01}^{2} \; = \;
\frac{t^{3/2} (x+1)}{\left(t^2-1\right) (t (3 x-1)+x-3)} \\
S_{02}^{2} \; = \; S_{03}^{2} \; = \; S_{04}^{2} \; = \; S_{05}^{2} \; = \; \frac{t^{3/2} (x-1)^3}{4 (t-1) \left(t x^2-1\right) (t (3 x-1)+x-3)} \\
S_{06}^{2} \; = \; S_{07}^{2} \; = \; S_{08}^{2} \; = \; S_{09}^{2} \; = \; \frac{t^{3/2} \left(x^2-1\right)}{4 \left(t^2-1\right) \left(t x^2+1\right)}
\end{array}
\label{Selements}
\ee
Evaluating \eqref{ZviaSmatrix} for $g=0$ with \eqref{Sxyt} gives
\begin{multline}
\label{S2S1xyt}
\grdim \HVW (S^2 \times S^1) = \\ =
\frac{2 t^{3/2} (x-1) y^{3/2} \left(x \left(t \left(x y \left(t \left(x^2+x+1\right) y-(t+1) x-x y\right)+y+1\right)-x+y-1\right)-1\right)}{\left(t^2-1\right) x^{3/2} \left(y^2-1\right) (t y-1) \left(t^2 x^2 y^2-1\right)}
= \\ = 2 t^{3/2} \left(
\frac{(x-1) y^{3/2} (x^2-x y+x+1)}{x^{3/2} (y-1) (y+1)}
+ O(t) \right)
\end{multline}
which turns into a more compact expression (\ref{VWbasic}b) when we specialize to $x=y$. Similarly, for general $g>0$ eqs. \eqref{ZviaSmatrix} and \eqref{Selements} lead to the claim in (\ref{VWbasic}c). The derivation of (\ref{VWbasic}a) is much simpler and will be presented shortly in section~\ref{sec:Qcohomology}.

It is curious to note that \eqref{S2S1xyt} naively has a pole of order 4 at $x=y=t=1$ (associated with 4 non-compact complex directions in the moduli space), whereas its simplified version  (\ref{VWbasic}b) has only an order-2 pole. This happens because of a partial cancellation between the numerator and the denominator in \eqref{S2S1xyt} and teaches us a useful lesson. Namely, if we were to multiply \eqref{S2S1xyt} by a factor $(t-1) (y-1) (t y-1) (t x y-1)$ that naively cancels the pole at generic $(x,y,t)$, we would get zero after a further specialization to $x=y=t=1$. Continuing along these lines, by a direct calculation it is not difficult to prove a general result:

\begin{lemma}\label{lemma:xtasymptotics}
At $x=1=t$, the asymptotic behavior of $\grdim \HVW \left( S^1 \times \Sigma_g \right)$ specialized to $y=x$, {\it i.e.} that of (\ref{VWbasic}b) and (\ref{VWbasic}c), is given by
\be
\grdim \HVW \left( S^1 \times \Sigma_g \right)
\; \sim \;
\begin{cases}
4 \left( 8 \frac{1-t}{1-x} \right)^{3g-3} , & \text{if } \; g > 1 \\
10 , & \text{if } \; g=1 \\
\frac{1}{(1-t) (1- t x^2)} , & \text{if } \; g=0
\end{cases}
\ee
\end{lemma}

Other limits and specializations of \eqref{ZviaSmatrix} can be analyzed in a similar fashion.

\begin{remark}
By construction, the graded dimension \eqref{ZviaSmatrix} of the Floer homology in Vafa-Witten theory reduces to the equivariant Verlinde formula for ordinary Higgs bundles with either $R=0$ or $R=2$ in suitable limits:
\begin{subequations}\label{xtlimits}
\be
R = 0: \qquad x \to 0, \quad y \to 0, \quad t = \mathrm{fixed}
\ee
\be
R = 2: \qquad x = \mathrm{fixed}, \quad y \to 0, \quad t \to 0
\ee
\end{subequations}
For example, in the case of genus $g=2$, the latter gives
\be
\frac{16 x^4 + 49 x^3 + 81 x^2 + 75 x + 35}{(1 - x^2)^3}
\; = \; 35+75 x+186 x^2+274 x^3+469 x^4
+ \ldots
\label{xlimit}
\ee
up to an overall factor $\left( \frac{x}{yt} \right)^{3/2}$ which has to do with the normalization of Vafa-Witten invariants. The expansion \eqref{xlimit} agrees with eq.(1.5) in \cite{MR3670728} at level $k= 4 = 0 + 2 + 2$. Note, although the Higgs fields with $R=0$ effectively disappear in the limit (\ref{xtlimits}b), they each leave a trace by shifting the value of $k$ by $+2$. For gauge groups of higher rank, the shift is by the dual Coxeter number, {\it cf.} \cite[sec. 5.1.1]{MR3670728}.
\end{remark}

More generally, for other values of $g$ similar expressions can be obtained by specializing \eqref{Sxyt} to $y=0$ and $t=0$:
\be
\begin{array}{l}
S_{00}^{2} \; = \; S_{01}^{2} \; = \;
\frac{(x-1) (x+1)^3}{x+3} \\
S_{02}^{2} \; = \; S_{03}^{2} \; = \; S_{04}^{2} \; = \; S_{05}^{2} \; = \; \frac{(x-1)^3}{4 (x+3)} \\
S_{06}^{2} \; = \; S_{07}^{2} \; = \; S_{08}^{2} \; = \; S_{09}^{2} \; = \; \tfrac{1}{4} (x^2-1)
\end{array}
\label{SeqVerlindelimit}
\ee
where we again omitted the overall factor $\left( \frac{x}{yt} \right)^{3/2}$ related to a choice of normalization. Since the values of $S_{0 \lambda}^2$ are pairwise equal, there are effectively 5 values of $\lambda$, in agreement with the fact that the equivariant Verlinde formula has $k+1$ Bethe vacua for general $k$.

Similarly, the limit (\ref{xtlimits}a) can be obtained by specializing \eqref{Selements} to $x=0$.

\begin{remark}\label{rem:decoration}
So far we tacitly ignored one important detail, which does not appear for $G=SU(2)$ but would enter the discussion for more general gauge groups. In general, the Vafa-Witten theory on $M_4$ requires a choice of a decoration ('t Hooft flux) valued in $H^2 (M_4; \Gamma) \cong \mathrm{Hom} \left( H^2 (M_4; \Z) , \Gamma \right)$, where
\be
\Gamma = \pi_1 (G)
\ee
On a 4-manifold of the form $M_4 = S^1 \times M_3$ this requires a choice of a decoration valued in $H^2 (M_3; \Gamma)$, which can be thought of as simply the restriction of the decoration from $M_4$, as well as the choice of grading valued in $H^2 (M_3; \widehat{\Gamma})$, where
\be
\widehat{\Gamma} = \mathrm{Hom} \left( \Gamma , U(1) \right)
\ee
is the Pontryagin dual group. Therefore, we conclude that, in general, $\HVW (M_3)$ is decorated by $H^2 (M_3; \Gamma)$ and graded by $H^2 (M_3; \widehat{\Gamma})$:
\begin{center}
\begin{tabular}{c|c}
	$\HVW (M_3)$ & structure \\
	\hline
	graded by &  $\phantom{\int^{\int}}$ $H^2 (M_3; \widehat{\Gamma})$ $\phantom{\int^{\int}}$  \\
	decorated by & $\phantom{\int^{\int}}$ $H^2 (M_3; \Gamma)$ $\phantom{\int^{\int}}$  \\
\end{tabular}
\end{center}
\end{remark}

We close this section by drawing a general lesson from preliminary considerations presented here. In a simple infinite family of 3-manifolds considered here, we found that the graded dimension of $\HVW (M_3)$ is always a power series, rather than a finite polynomial. In fact, already from the preliminary analysis in this section one can see several good reasons why $\HVW (M_3)$ is expected to be infinite-dimensional for a general 3-manifold, a conclusion that will be further supported by considerations in section~\ref{sec:Qcohomology}.
\begin{conjecture}\label{conj:infdimnl}
$\HVW (M_3)$ is infinite-dimensional for any closed 3-manifold $M_3$.
\end{conjecture}
In light of this Conjecture, the role of the equivariant parameters $(x,y,t)$ --- that were central to the considerations of the present section --- is to provide a way to regularize the infinity in $\dim \HVW (M_3)$. This also illustrate well the challenge of computing $\HVW (M_3)$ on more general 3-manifolds: in the absence of suitable symmetries and equivariant parameters, one has to work with the entire $\HVW (M_3)$, which is infinite-dimensional.

\section{Q-cohomology}
\label{sec:Qcohomology}

In this section, we pursue the same goal --- to explicitly compute $\HVW (M_3)$ for a class of 3-manifolds --- via a direct analysis of $Q$-cohomology in Vafa-Witten theory. The results agree with the preliminary considerations in section~\ref{sec:eqVerlinde}.

Recall, that the off-shell realization \cite{Vafa:1994tf} of Vafa-Witten theory involves the following set of fields\footnote{Here, the subscript indicates the degree of the differential form on a 4-manifold, while the superscript is the $U(1)_t$ grading, also known as the ``ghost number.''}
\be
\begin{array}{lcl}
	\multicolumn{2}{c}{~~\text{{\bf bosons:}}~~} \\[.2cm]
	\phi^{+2},~{\bar \phi}^{\,-2},~C^0 & : \quad & \text{scalars (0-forms)} \\
	A^0_1,~\tilde H^0_1 & : \quad & \text{1-forms} \\
	(B^+_2)^0,~(D^+_2)^0 & : \quad & \text{self-dual 2-forms}
\end{array}
\qquad
\begin{array}{lcl}
	\multicolumn{2}{c}{~~\text{{\bf fermions}}~~} \\[.2cm]
	\zeta^{+1}, ~\eta^{-1} & : \quad & \text{scalars (0-forms)} \\
	\psi^{+1}_1, ~\tilde \chi_1^{-1} & : \quad & \text{1-forms} \\
	(\tilde \psi^+_2)^{+1}, ~(\chi^+_2)^{-1} & : \quad & \text{self-dual 2-form}
\end{array}
\notag
\ee
and their $Q$-transformations:
\be
\begin{aligned}
	Q A & = \psi_1 \\
	Q \phi & = 0 \\
	Q \bar \phi & = \eta \\
	Q \eta & = i [\bar \phi , \phi] \\
	Q \psi_1 & = d_A \phi \\
	Q \chi^+_2 & = D^+_2 + s_2^+ \\
	Q D^+_2 & = i [\chi^+_2, \phi] - Q s_2^+
\end{aligned}
\qquad\qquad
\begin{aligned}
	Q B^+_2 & = \tilde \psi^+_2 \\
	Q \tilde \psi^+_2 & = i [B^+_2 , \phi] \\
	Q \tilde \chi_1 & = \tilde H_1 + s_1 \\
	Q \tilde H_1 & = i [\tilde \chi_1 , \phi] - Q s_1 \\
	Q C & = \zeta \\
	Q \zeta & = i [C, \phi]
\end{aligned}
\label{BRST}
\ee
where
$$
s_2^+ = F^+_{\alpha \beta} + [B^+_{\gamma \alpha} , B^{+\gamma}_{\beta}] + 2i [B^+_{\alpha \beta} , C]
$$
$$
s_1 = D_{\alpha \dot \alpha} C + i D_{\beta \dot \alpha} {B^{+\beta}}_{\alpha}
$$
The on-shell formulation is obtained simply by setting $D^+_{\alpha \beta} = 0 = \tilde H_{\alpha \dot \alpha}$.

\begin{lemma}
Up to gauge transformations, $Q^2 = 0$.
\end{lemma}

\begin{proof}
This is easily demonstrated by a direct calculation.
\end{proof}

This basic fact about topological gauge theory is the reason one can define the $Q$-cohomology groups
\be
\HVW (M_3) := \frac{\text{ker} \; Q}{\text{im} \; Q}
\label{HVWviaQ}
\ee
which are the main objects of study in the present paper.
\parpic[r]{
	\begin{minipage}{40mm}
		\centering
		\includegraphics[scale=0.15]{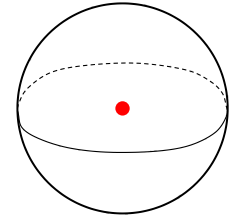}
		\captionof{figure}{The space of local operators is isomorphic to the space of states on a sphere.}
		\label{fig:S2ops}
	\end{minipage}%
}
Just like the fields of Vafa-Witten theory are differential forms of various degrees graded by weights of the $U(1)_t$ symmetry, so are the $Q$-cohomology classes. The cohomology classes represented by differential forms of degree 0 are called {\it local operators}, {\it i.e.} operators supported at points on a 4-manifold. Since in four dimensions a link of a point is a 3-sphere, the space of local operators is naturally isomorphic to $\HVW (S^3)$, {\it cf.} Figure~\ref{fig:S2ops}. This relation, often called the state-operator correspondence, has obvious generalizations that will be discussed below.

For example, when we talk about $M_3 = S^2 \times S^1$ we are effectively counting local operators in three-dimensional theory obtained by dimensional reduction of the 4d theory on a circle. Such 3d local operators come either from local operators in four dimensions or from 4d line operators, {\it i.e.} $Q$-cohomology classes supported on lines (a.k.a. 1-observables). We will return to the local operators of such 3d theory shortly, after discussing the original 4d theory first, thus providing a proof of (\ref{VWbasic}a).

\begin{proposition}
\label{prop:HS3}
In the notations \eqref{Ttower} introduced in the Introduction, the space of local observables in Vafa-Witten theory with gauge group $G=SU(2)$, {\it i.e.} the space of states on $M_3 = S^3$, is
\be
\HVW (S^3) = \CT^+_0  \cong \C [u]
\label{HS3}
\ee
generated by $u = \Tr \phi^2$.
\end{proposition}

Moreover, \eqref{HS3} transforms trivially under the modular $SL(2,\Z)$ action discussed in section~\ref{sec:eqVerlinde}.

\begin{proof}
To construct a local operator, we can only use 0-forms. We can not use their exterior derivatives or forms of higher degrees because that would require the metric and the resulting operators would be $Q$-exact.
This limits our arsenal to the 0-forms $\phi$, ${\bar \phi}$, $C$, $\zeta$, and $\eta$. In addition, all observables (not only local) must be gauge-invariant. Since all fields in Vafa-Witten theory transform in the adjoint representation of the gauge group, this means we need to consider traces of polynomials in $\phi$, ${\bar \phi}$, $C$, $\zeta$, and $\eta$. Inspecting the $Q$-action \eqref{BRST} on these fields leads to $u = \Tr \phi^2$ as the only independent gauge-invariant local observable, as also found {\it e.g.} in \cite{Lozano:1999ji}. Polynomials in $u$ also represent $Q$-cohomology classes, of course. This leads to \eqref{HS3}.
\end{proof}

In higher rank, $\HVW (S^3)$ is spanned by invariant polynomials of $\phi$. In the opposite direction, it is also instructive to consider a version of the Proposition~\ref{prop:HS3} in Vafa-Witten theory with gauge group $G=U(1)$. The arguments are similar, but the $Q$-action on fields is much simpler than \eqref{BRST}. Namely, in abelian theory commutators vanish and $d_A$ becomes the ordinary exterior derivative:
\be
\begin{aligned}
	Q A & = \psi_1 \\
	Q \phi & = 0 \\
	Q \bar \phi & = \eta \\
	Q \eta & = 0 \\
	Q \psi_1 & = d \phi \\
	Q \chi^+_2 & = D^+_2 + s_2^+ \\
	Q D^+_2 & = 0
\end{aligned}
\qquad\qquad
\begin{aligned}
	Q B^+_2 & = \tilde \psi^+_2 \\
	Q \tilde \psi^+_2 & = 0 \\
	Q \tilde \chi_1 & = \tilde H_1 + s_1 \\
	Q \tilde H_1 & = 0 \\
	Q C & = \zeta \\
	Q \zeta & = 0
\end{aligned}
\label{BRSTabelian}
\ee
We quickly learn that in Vafa-Witten theory with $G=U(1)$ one also has $\HVW (S^3) = \CT^+_0$, just as \eqref{HS3} in the theory with $G=SU(2)$. Moreover, both of these agree with the Floer homology of $S^3$ and with the Heegaard Floer homology of $S^3$, {\it cf.} (\ref{HFbasic}a). This is not a coincidence and is a good illustration of a deeper set of relations between $\HVW (M_3)$ and the instanton Floer homology of the same 3-manifold that will be discussed further below.
These parallels with the instanton Floer homology will help us to better understand $\HVW (M_3)$ for $M_3 = \Sigma_g \times S^1$.

\subsection{Comparison to Floer homology}

The first column in \eqref{BRST} is precisely the $Q$-cohomology in Donaldson-Witten theory, with the same action of $Q$. This suggests that we should consider $\HVW (M_3)$ and $HF (M_3)$ in parallel, anticipating a general relation of the form
\be
HF (M_3) \subseteq \HVW (M_3)
\label{HFinHVW}
\ee
where the reduction of grading mod 8 is understood on the right-hand side. Clarifying the role of such relation, and understanding under what conditions it should be expected, is one of the motivations in the discussion below. It will help us to understand better the structure of $\HVW (M_3)$ for general $M_3$ and for $M_3 = \Sigma_g \times S^1$ in particular.

Recall, that the Floer homology $HF (\Sigma_g \times S^1)$ is isomorphic to the quantum cohomology of $\mathrm{Bun}_{G_C}$, the space of holomorphic bundles on $\Sigma_g$ (see {\it e.g.} \cite{MR1670396}): 
\be
HF^* (\Sigma_g \times S^1) \; \cong \; QH^* (\mathrm{Bun}_{G_C})
\label{HFQH}
\ee
which, in turn, is isomorphic to the space of flat $G$-connections, $\mathrm{Bun}_{G_C} \cong \CM_{\text{flat}} (\Sigma_g, G)$, according to the celebrated theorem of Narashimhan and Seshadri. More precisely, here by $\mathrm{Bun}_{G_C}$ we mean the space of bundles with fixed determinant, such that it is simply-connected for any $g$ and $\mathrm{Bun}_{G_C} = \mathrm{pt}$ when $g=1$. The Poincar\'e polynomial of $\mathrm{Bun}_{G_C} (\Sigma_g)$ is given by the Harder-Narasimhan formula:
\be
P (\mathrm{Bun}_{SL(2)} (\Sigma_g)) = \frac{(1 + t^3)^{2g} - t^{2g} (1+t)^{2g}}{(1-t^2) (1-t^4)}
\ee
If we wish to work with all bundles (rather than bundles with fixed determinant), which in gauge theory language corresponds to replacing $G=SU(N)$ by $G=U(N)$, then
\be
H^* (\mathrm{Bun}_{GL(N)} (\Sigma_g)) \cong H^* (T^{2g}) \otimes H^* (\mathrm{Bun}_{SL(N)} (\Sigma_g))
\ee

Returning to the quantum cohomology \eqref{HFQH} and the corresponding Floer homology, it has the following generators:
\begin{eqnarray}
\alpha \in HF^2 (\Sigma_g \times S^1) & \nonumber \\
\psi_i \in HF^3 (\Sigma_g \times S^1) & \quad 1 \le i \le 2g \label{HFggenerators} \\
\beta \in HF^4 (\Sigma_g \times S^1) & \nonumber
\end{eqnarray}
The action of $\mathrm{Diff} (\Sigma_g)$ on $HF^* (\Sigma_g \times S^1)$ factors through $Sp(2g,\Z)$ on $\psi_i$, so that the invariant part
\be
HF^* (\Sigma_g \times S^1)^{Sp(2g,\Z)} \; = \; \C [\alpha, \beta, \gamma] / J_g
\ee
is generated by $\alpha$, $\beta$, and the $Sp (2g,\Z)$-invariant combination
\be
\gamma = -2 \sum_{i=1}^g \psi_i \psi_{i+g}
\ee
Moreover,
\be
HF^* (\Sigma_g \times S^1) \; = \;
\bigoplus_{k=0}^g \Lambda_0^k HF^3 \otimes \C [\alpha, \beta, \gamma] / J_{g-k}
\ee
where $\Lambda_0^k HF^3 = \ker (\gamma^{g-k+1}: \Lambda^k HF^3 \to \Lambda^{2g-k+2} HF^3)$ is the primitive part of $\Lambda^k HF^3$ and the explicit description of $J_g$ can be found {\it e.g.} in \cite{MR1670396}.

Note, $p$-form observables in Donaldson-Witten theory correspond to cohomology classes of homological degree $4-p$. This can be understood as a consequence of the standard descent procedure,
\begin{eqnarray}
0 \; = \; i \{ Q, W_0 \} & \qquad , \qquad &
dW_0 \; = \; i \{ Q, W_1 \} \nonumber \\
dW_1 \; = \; i \{ Q, W_2 \} & \qquad , \qquad &
dW_2 \; = \; i \{ Q, W_3 \} \label{descentrels} \\
dW_3 \; = \; i \{ Q, W_4 \} & \qquad , \qquad &
dW_4 \; = \; 0 \nonumber
\end{eqnarray}
applied to the local observable $W_0 = u = \Tr \phi^2$ that we already met earlier in \eqref{HS3} and that also is a generator of $HF (S^3) \cong \HVW (S^3)$, {\it cf.} \eqref{HFinHVW}. Indeed, in the conventions such that the homological grading in \eqref{HFggenerators} is twice the $U(1)_t$ degree, $\beta = - 4u$ has $U(1)_t$ degree $2$ and the topological supercharge $Q$ carries $U(1)_t$ degree $+\frac{1}{2}$. In this conventions, $W_p$ constructed as in \eqref{descentrels} is a $p$-form on $M_4$ of $U(1)_t$ degree
\be
\deg_t (W_p) = \deg_t (W_0) - \frac{p}{2}
\ee
Therefore, integrating $W_p$ over a $p$-cycle $\gamma$ in $M_4$ we obtain a topological observable
with $U(1)_t$ grading $\deg_t (W_0) - \frac{p}{2}$,
\be
\mathcal{O}^{(\gamma)} \; := \; \int_{\gamma} W_{\dim (\gamma)}
\label{pobservable}
\ee
Since in the conventions used here the homological grading and $U(1)_t$ grading differ by a factor of 2, the homological degree of the observable $\mathcal{O}^{(\gamma)}$ is $2 \deg_t (W_0) - p$.

The relation \eqref{HFQH} has an analogue in the Vafa-Witten theory and can be understood in the general framework {\it a la} Atiyah-Floer. Indeed, the push-forward, or the ``fiber integration,'' of the 4d TQFT functor along $\Sigma_g$ gives a 2d TQFT, namely the A-model with target space given by the space of solutions to gauge theory PDEs on $\Sigma_g$. When this process, called topological reduction \cite{Bershadsky:1995vm}, is applied to the Donaldson-Witten theory it gives precisely A-model with target space $\mathrm{Bun}_{G_C}$. This is in excellent agreement with the Atiyah-Floer conjecture which, among other things, asserts that upon such fiber integration (or, topological reduction) a 3-manifold $M_3^b$ with boundary $\Sigma_g = \partial (M_3^b)$ defines a boundary condition (``brane'') $\CB (M_3^b)$ in the A-model with target space $\mathrm{Bun}_{G_C} (\Sigma_g)$. In this way, the instanton Floer homology of the Heegaard decomposition
\be
M_3 = M_3^+ \cup_{\Sigma_g} M_3^-
\label{Heegaard}
\ee
can be understood as the Lagrangian Fukaya-Floer homology of $\CB (M_3^+)$ and $\CB (M_3^-)$ in the ambient moduli space $\mathrm{Bun}_{G_C} \cong \CM_{\text{flat}} (\Sigma_g, G)$. In this relation, gauge instantons that provide a differential in the Floer complex become disk instantantons in the A-model on $\mathrm{Bun}_{G_C} (\Sigma_g)$.
Similarly, in the Vafa-Witten theory the space of solutions to the PDE on $\Sigma_g$, {\it i.e.} target space of the topological sigma-model, $\MVW (\Sigma_g, G)$, is the space of $\CE$-valued $G$-Higgs bundles on $\Sigma$ that we already encountered in section \ref{sec:eqVerlinde}. There are some notable differences, however.

Thus, one important novelty of the Vafa-Witten theory is that it has much larger configuation space (space of fields), larger (super)symmetry and, correspondingly, larger structure in the push-forward to the topological sigma-model along $\Sigma_g$. In particular, the resulting sigma-model has no disk instantons ({\it cf.} {\it e.g.} \cite[Lemma 15]{MR1949786}) since in the context of Vafa-Witten theory $\CB (M_3^+)$ and $\CB (M_3^-)$ are holomorphic Lagrangians submanifolds of $\MVW (\Sigma_g, G)$. This strongly suggests that the original 4d gauge theory also has no instantons that contribute to $\HVW (M_3)$.

This conclusion is also supported by the fact that on 
$\R \times M_3$ or on $S^1 \times M_3$ Vafa-Witten theory is equivalent to another, amphicheiral twist of $\CN=4$ super-Yang-Mills \cite{Geyer:2001yc} which is known to have no instantons \cite{Labastida:1997vq}. Based on all of these, we expect:

\begin{conjecture}\label{conj:noinstantons}
In Vafa-Witten theory on 3-manifolds, there are no (disk) instantons that contribute to $\HVW (M_3)$.
\end{conjecture}

In what follows we assume the validity of this conjecture which drastically simplifies the computation of the homology groups $\HVW (M_3)$. It basically means that the chain complex underlying this homology theory is obtained by restricting the original configuration space to the space of fields on $M_3$, with the differential induced by the action of $Q$ in \eqref{BRST}, justifying \eqref{HVWviaQ}.

As it often happens in topological sigma-models, the action of $Q$ can be interpreted geometrically as the suitable differential acting on the differential forms on the target manifold. In the context of Vafa-Witten theory, where the target space $\MVW (\Sigma_g, G)$ is the space of $\CE$-valued $G$-Higgs bundles on $\Sigma$, this also clarifies and further justifies the claim in Conjecture~\ref{conj:infdimnl}. Namely, from the cohomological perspective discussed here, the infinite-dimensionality of $\HVW (M_3)$ is attributed to the non-compactness of $\MVW (\Sigma_g, G)$. Just like the space of holomorphic functions on $\C$ is infinite-dimensional, the non-compactness of $\MVW (\Sigma_g, G)$ leads to $\dim \HVW (M_3) = \infty$ for any $M_3$. This analogy is, in fact, realized in the Vafa-Witten theory with gauge group $G=U(1)$ that we already briefly discussed around eq.\eqref{BRSTabelian} in this section.

The computation of $\HVW (M_3)$ for $M_3 = \Sigma_g \times S^1$ is similar to the computation of \eqref{HS3}, except many other fields besides $\phi$ have zero-modes on $M_3$ and contribute to $\HVW (M_3)$. Equivalently, as explained around Figure~\ref{fig:S2ops}, $p$-form observables in the original theory integrated over $p$-cycles in $M_3$ give rise to non-trivial $Q$-cohomology classes, {\it cf.} \eqref{pobservable}.
The counting of zero-modes is especially simple in the case of abelian theory, on which more general consideration can be modelled.
For example, when $G=U(1)$, 1-form observables include $\psi_1$, which can be integrated over 1-cycles and give $2g+1$ Grassmann (odd) zero-modes on $M_3 = \Sigma_g \times S^1$. More precisely, the zero-modes of $A$ parametrize $\CM_{\text{flat}} (\Sigma_g, U(1)) = \mathrm{Jac} (\Sigma_g) \cong T^{2g}$, the fields $\phi$, $\bar \phi$, $C$, $\eta$, and $\zeta$ have one zero-mode each, while each of the remaining fields has $2g+1$ zero-modes, modulo constraints (field equations). Altogether, these modes parametrize\footnote{Notice that bosonic (even) and Grassmann (odd) dimensions of this superspace are equal; this is a consequence of the fact that Vafa-Witten theory is {\it balanced} \cite{Dijkgraaf:1996tz}. The bosonic (even) part of this superspace is $\C^2 \times \CM_{\text{flat}} (M_3, U(1)_{\C})$, and the same expression holds for general 3-manifold $M_3$.}
\be
\MVW (\Sigma_g, U(1)) \; = \; T^* \mathrm{Jac} (\Sigma_g) \times \Pi \C \times \Pi \C \times \left( \C \times \Pi \C^g \right) \times \left( \C \times \Pi \C^g \right)
\label{MVWabelian}
\ee
where $\Pi \C^n$ represents Grassmann (odd) space and contributes to $Q$-cohomology a tensor product of $n$ copies of the fermionic Fock space, $\CF = \Lambda^* [\xi] \cong H^* (\cp^1)$. In comparison, each copy of $\C$ in the moduli space \eqref{MVWabelian} contributes to the $Q$-cohomology a factor of $\CT^+$, the Fock space of a single boson described in \eqref{Ttower} and illustrated in Figure~\ref{fig:spectrum}.

The result \eqref{MVWabelian} agrees with the analysis in section~\ref{sec:eqVerlinde}, where it can be understood as the moduli space of $\CE$-valued $G$-Higgs bundles on $\Sigma$, with $G=U(1)$. 
Indeed, a Higgs field on $\Sigma$ with $R$-charge $R$ ({\it cf.} \eqref{EVW}) contributes $H^0 (K_{\Sigma}^{R/2})$ bosonic zero-modes and $H^0 (K_{\Sigma}^{1-R/2})$ fermionic zero-modes, all valued in $\mathrm{Lie} (G)$. In the case of Vafa-Witten theory we have three Higgs fields with $R=2$, 0, and 0, respectively, which leads precisely to \eqref{MVWabelian}. Furthermore, in the notations of section~\ref{sec:eqVerlinde}, the symmetry $U(1)_x$ acts on the fiber of $T^* \mathrm{Jac} (\Sigma)$ and one of the $\Pi \C$ factors in \eqref{MVWabelian}, whereas symmetries $U(1)_y$ and $U(1)_t$ each act on the corresponding copy of $\left( \C \times \Pi \C^g \right)$ in \eqref{MVWabelian}.
For a non-abelian $G$, the analysis is similar, though $\MVW (\Sigma_g, G)$ is no longer a product {\it a la} \eqref{MVWabelian}; $T^* \mathrm{Jac} (\Sigma)$ is replaced by $\Hom \left( \pi_1 (\Sigma) , G_{\C} \right)$ and each copy of $\C$ (resp. $\Pi \C$) is replaced by $\mathfrak{g}_{\C}$ (resp. $\Pi \mathfrak{g}_{\C}$), subject to the constraints and gauge transformations that act simultaneously on all of the factors.

Note, unlike \eqref{HS3}, the cohomology of \eqref{MVWabelian} and its non-abelian generalization $\HVW (\Sigma_g \times S^1)$ transform non-trivially under the modular $SL(2,\Z)$ action discussed in section~\ref{sec:eqVerlinde}. In particular, the $S$ element of $SL(2,\Z)$ acts on $\mathrm{Jac} (\Sigma_g)$ as the Fourier-Mukai transform.

As already noted earlier, the spectrum of fields and the (super)symmetry algebra in the Vafa-Witten theory is much larger compared to what one finds in the Donaldson-Witten theory, {\it cf.} \eqref{HFinHVW}. As a result, there is more structure in $Q$-cohomology of Vafa-Witten theory, to which we turn next.

\subsection{Differentials and spectral sequences}

There are two ways in which spectral sequences typically arise in a cohomological TQFT: one can change the differential $Q$ while keeping the theory intact, or one can deform the theory. (See \cite{Gukov:2015gmm} for an extensive discussion and realizations in various dimensions.)
In the first case, we obtain a different cohomological invariant in the same theory, whereas in the latter case we obtain a relation between cohomological invariants of two different theories.

In the present context of Vafa-Witten theory, a natural class of deformations consists of relevant deformations that, in a physical theory, initiate RG flows to new conformal fixed points. If we want the resulting SCFT to allow a topological twist on general 4-manifolds, we need to preserve at least $\CN=2$ supersymmetry of the physical theory and the RG-flow. (This condition is necessary, but may not be sufficient; there can be further constraints.) In such a scenario, one can expect the following:

\begin{conjecture}\label{conj:RGflows}
A 4d $\CN=2$ theory that can be reached from 4d $\CN=4$ super-Yang-Mills via an RG-flow leads to a spectral sequence that starts with $\HVW (M_3)$ and converges to Floer-like homology in that $\CN=2$ theory.
\end{conjecture}

One interesting feature of spectral sequences induced by RG-flows is that a local relevant operator $\mathcal{O}$ that triggers the flow may transform non-trivially under the subgroup of $SO(6)_R$ that becomes the $R$-symmetry of the 4d $\CN=2$ SCFT. If so, under the topological twist it may no longer remain a scalar (a 0-form) on $M_4$, thus requiring a choice of additional structure. For example, a mass deformation to $\CN=2^*$ theory, already considered in \cite{Vafa:1994tf,Labastida:1997xk}, makes all the fields in the right column of \eqref{BRST} massive\footnote{At the level of $Q$-cohomology, it modifies the right-hand side in \eqref{BRST}.} and does not require any additional choices or structures on $M_3$ or $M_4$. It initiates an RG-flow to 4d $\CN=2$ super-Yang-Mills, whose topologically twisted version is the Donaldson-Witten theory, providing another perspective on the connection between these two topological theories, {\it cf.} \eqref{HFinHVW}.

Now, let us consider the other mechanism that leads to spectral sequences, in which the theory remains unchanged, but the definition of $Q$ changes. This is only possible if a theory admits more than one BRST operator (scalar supercharge) that squares to zero. Luckily, the Vafa-Witten theory is a good example; in addition to the original differential $Q$ it has the second differential, $Q'$, that acts as follows
\be
\begin{aligned}
	Q' A & = \tilde \chi_1 \\
	Q' \phi & = \zeta \\
	Q' \bar \phi & = 0 \\
	Q' \eta & = i [C, \bar \phi] \\
	Q' \psi_1 & = - \tilde H_1 - s_1 + d_A C \\
	Q' \chi^+_2 & = i [B^+_2 , \bar \phi] \\
	Q D^+_2 & = i [\chi^+_2, \phi] - Q s_2^+
\end{aligned}
\qquad\qquad
\begin{aligned}
	Q' B^+_2 & = \chi^+_2 \\
	Q \tilde \psi^+_2 & =  \\
	Q' \tilde \chi_1 & = - d_A \bar \phi \\
	Q \tilde H_1 & = i [\tilde \chi_1 , \phi] - Q s_1 \\
	Q' C & = - \eta \\
	Q' \zeta & = i [\bar \phi , \phi]
\end{aligned}
\label{BRSTprime}
\ee

Since Vafa-Witten theory on a general 4-manifold has $SU(2)_R$ symmetry, under which $Q$ and $Q'$ transform as a two-dimensional representation, ${\bf 2}$, it is convenient to write the action of $Q$ and $Q'$ in a way that makes this symmetry manifest.
Therefore, we introduce $Q^a = ( Q, Q' )$, where $a=1,2$.
Similarly, we can combine all odd (Grassmann) fields of the Vafa-Witten theory into three $SU(2)_R$ doublets:
$\zeta$ and $\eta$ into a doublet of 0-forms $\eta^a$,
$\psi_1$ and $\tilde \chi_1$ into a doublet of 1-forms $\psi_1^a$,
$\tilde \psi_2^+$ and $\chi_2^+$ into a doublet of self-dual 2-forms $\chi_2^a$.

The bosonic (or, even) fields of the Vafa-Witten theory likewise combine into a triplet of 0-forms $\phi^{ab} = (\phi, C, \bar \phi)$, and the rest are $SU(2)_R$ singlets: $A$, $D_2^+$, $B_2^+$, and $\tilde H_1$. Then, the action of the BRST differentials \eqref{BRST} and \eqref{BRSTprime} can be written in a more compact form:
\be
\begin{aligned}
	Q^a A & = \psi_1^a \\
	Q^a \phi^{bc} & = \tfrac{1}{2} \epsilon^{ab} \eta^c + \frac{1}{2} \epsilon^{ac} \eta^b \\
	Q^a \psi_1^b & = d_A \phi^{ab} + \epsilon^{ab} \tilde H_1 \\
	Q^a \chi^b_2 & = [B_2^+, \phi^{ab}] + \epsilon^{ab} G_2^+
\end{aligned}
\qquad\qquad
\begin{aligned}
	Q^a B^+_2 & = \chi_2^a \\
	Q^a \eta^b & = - \epsilon_{cd} [\phi^{ac}, \phi^{bd}] \\
	Q^a \tilde H_1 & = - \tfrac{1}{2} d_A \eta^a - \epsilon_{cd} [\phi^{ac} , \psi_1^d] \\
	Q^a G_2^+ & = - \tfrac{1}{2} [B_2^+ , \eta^a] - \epsilon_{bc} [\phi^{ab} , \chi_2^c]
\end{aligned}
\label{scalarQ}
\ee
where $\epsilon_{ab}$ is the invariant tensor of $SU(2)_R$, and we use the conventions $\epsilon_{12} = 1$, $\epsilon^{ac} \epsilon_{cb} = - \delta^a_b$, $\varphi_a = \varphi^b \epsilon_{ba}$, $\varphi^a = \epsilon^{ab} \varphi_b$.

In addition to the differentials $Q^a = ( Q, Q' )$, the Vafa-Witten theory also has a doublet of vector supercharges (described in detail in Appendix~\ref{sec:appendix}).
On a 4-manifold of the form \eqref{M4RM3} they produce another doublet of BRST differentials, $\bar Q^a$, that are scalars with respect to the holonomy group of $M_3$.
Altogether, the total number of scalar supercharges is $N_T = 4$, as noted earlier {\it e.g.} in Table~\ref{tab:symmetries}.
In order to write their action on fields in Vafa-Witten theory, it is convenient to describe the latter as forms on $M_3$.

Since $\Omega^0 (\R \times M_3) \cong \Omega^0 (M_3)$, all 0-forms (scalars) remain 0-forms on $M_3$.
In other words, $\phi^{ab}$ and $\eta^a$ are not affected by the reduction to $M_3$.
In the case of 1-forms, we have $\Omega^1 (\R \times M_3) \cong \Omega^1 (M_3) \oplus \Omega^0 (M_3)$, and so $A$, $\tilde H_1$, $\psi_1^a$ produce additional 0-forms on $M_3$ that, following \cite{Geyer:2001yc}, we denote $\rho$, $Y$, and $\bar \eta^a$, respectively.
Finally, using $\Omega^{2,+} (\R \times M_3) \cong \Omega^1 (M_3)$, we can replace the self-dual forms $B_2^+$, $G_2^+$, $\chi_2^a$ by 1-forms on $M_3$: $V_1$, $\bar B_1$, and $\bar \psi_1^a$, respectively. It is also convenient to denote $\tilde H_1 + [V_1 , \rho] = B_1$.
Then, the action of all four BRST operators can be written as
\be
\begin{aligned}
Q^a A & = \psi_1^a \\
Q^a \phi^{bc} & = \tfrac{1}{2} \epsilon^{ab} \eta^c + \tfrac{1}{2} \epsilon^{ac} \eta^b \\
Q^a \eta^b & = - \epsilon_{cd} [\phi^{ac} , \phi^{bd}] \\
Q^a \psi_1^b & = d_A \phi^{ab} - \epsilon^{ab} [V_1 , \rho] + \epsilon^{ab} B_1 \\
Q^a B_1 & = -\tfrac{1}{2} d_A \eta^a + \tfrac{1}{2} [V , \bar \eta^a ] \\
& ~- \epsilon_{cd} [\phi^{ac} , \psi_1^d] - [\rho , \bar \psi_1^a ] \\
Q^a V_1 & = \bar \psi_1^a \\
Q^a \rho & = \tfrac{1}{2} \bar \eta^a \\
Q^a \bar \eta^b & = 2 [\rho , \phi^{ab}] + \epsilon^{ab} Y \\
Q^a \bar \psi^b_1 & = [V_1 , \phi^{ab}] + \epsilon^{ab} d_A \rho + \epsilon^{ab} \bar B_1 \\
Q^a \bar B_1 & = - \tfrac{1}{2} d_A \bar \eta^a - \tfrac{1}{2} [V_1 , \eta^a] \\
& ~- \epsilon_{cd} [\phi^{ac} , \bar \psi_1^d ] + [\rho , \psi^a_1] \\
Q^a Y & = - [\rho , \eta^a] - \epsilon_{cd} [\phi^{ac} , \bar \eta^d]
\end{aligned}
\qquad\qquad
\begin{aligned}
\bar Q^a A & = \bar \psi_1^a \\
\bar Q^a \phi^{bc} & = \tfrac{1}{2} \epsilon^{ab} \bar \eta^c + \tfrac{1}{2} \epsilon^{ac} \bar \eta^b \\
\bar Q^a \bar \eta^b & = - \epsilon_{cd} [\phi^{ac} , \phi^{bd}] \\
\bar Q^a \bar \psi_1^b & = d_A \phi^{ab} - \epsilon^{ab} [V_1 , \rho] - \epsilon^{ab} B_1 \\
\bar Q^a B_1 & = \tfrac{1}{2} d_A \bar \eta^a + \tfrac{1}{2} [V , \eta^a ] \\
& ~+ \epsilon_{cd} [\phi^{ac} , \bar \psi_1^d] - [\rho , \psi_1^a ] \\
\bar Q^a V_1 & = - \psi_1^a \\
\bar Q^a \rho & = - \tfrac{1}{2} \eta^a \\
\bar Q^a \eta^b & = - 2 [\rho , \phi^{ab}] - \epsilon^{ab} Y \\
\bar Q^a \psi^b_1 & = - [V_1 , \phi^{ab}] - \epsilon^{ab} d_A \rho + \epsilon^{ab} \bar B_1 \\
\bar Q^a \bar B_1 & = - \tfrac{1}{2} d_A \eta^a + \tfrac{1}{2} [V_1 , \bar \eta^a] \\
& ~- \epsilon_{cd} [\phi^{ac} , \psi_1^d ] - [\rho , \bar \psi^a_1] \\
\bar Q^a Y & = - [\rho , \bar \eta^a] + \epsilon_{cd} [\phi^{ac} , \eta^d]
\end{aligned}
\label{allQ}
\ee
Modulo gauge transformations, these operators obey
\be
\{ Q^a , Q^b \} = 0
\,, \qquad
\{ Q^a , \bar Q^b \} = - \epsilon^{ab} H
\,, \qquad
\{ \bar Q^a , \bar Q^b \} = 0
\ee
where the ``Hamiltonian'' $H$ generates translations along $\R$ in $M_4 = \R \times M_3$, {\it cf.} \eqref{M4RM3}.

Note, this algebra has the same structure as the familiar 2d $\CN=(2,2)$ supersymmetry algebra \cite{MR2003030}; namely, it is the 1d version of this superalgebra obtained by reduction to supersymmetric quantum mechanics. Along the same lines, the subalgebra generated by $Q^a$ should be compared to the 2d $\CN=(0,2)$ supersymmetry algebra.
Only this subalgebra is relevant to defining homological invariants of 3-manifolds that extend to a four-dimensional TQFT-like structure. However, from a purely three-dimensional perspective, one might consider a more general combination of BRST differentials
\be
d = s_a Q^a + r_b \bar Q^b
\ee
Then, a simple calculation gives
\be
d^2 = - \epsilon^{ab} s_a r_b H = (s_2 r_1 - s_1 r_2) H
\ee
and the vanishing of the right-hand side defines a quadric in the $\cp^3$, parametrized by $(s_a , r_a)$ modulo the overall scale, {\it cf.} \cite{Gukov:2015gmm}. It is $S^2 \times S^2$, which therefore is the space of possible choices of the BRST differential in the Vafa-Witten theory on a general 3-manifold. It would be interesting to analyze further how the computation of $Q$-cohomology varies over $S^2 \times S^2$, if at all.

\section{Applications and future directions}
\label{sec:final}

We conclude with a brief discussion of various potential applications of $\HVW (M_3)$ and directions for future work.

\parpic[r]{
	\begin{minipage}{40mm}
		\centering
		\includegraphics[scale=2.3,trim={0.2cm 0.2cm 2.4cm 0.2cm},clip]{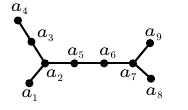}
		\captionof{figure}{An example of a plumbing graph.}
		\label{fig:plumbing}
	\end{minipage}%
}
\addtocontents{toc}{\protect\setcounter{tocdepth}{1}}
\subsection*{More general 3-manifolds}

One clear goal that also motivated this work is the computation of $\HVW (M_3)$ for more general 3-manifolds. The infinite family considered in this paper is part of a more general family of Seifert 3-manifolds which, in turn, belongs to a larger family of plumbed 3-manifolds. The latter can be conveniently described by a combinatorial data, as illustrated in Figure~\ref{fig:plumbing}, and therefore it would be nice to formulate $\HVW (M_3)$ directly in terms of such combinatorial data.

\subsection*{Mapping tori}

In general, there could be various ways of defining $\HVW (M_3)$, but all such versions should enjoy the action of the mapping class group $\mathrm{MCG} (M_3)$, {\it cf.} \eqref{MCGaction}. Studying this action more systematically and computing the invariants of the mapping tori of $M_3$,
\be
M_4 = \frac{M_3 \times I}{(x,0) \sim (\varphi (x), 1)}
\ee
as $\Tr_{\HVW (M_3)} \varphi$ could help to identify the particular definition of $\HVW (M_3)$ that matches the existing techniques of computing 4-manifold invariants $\ZVW (M_4)$. This could have important implications for understanding the functoriality in the Vafa-Witten theory, and developing the corresponding TQFT-like structure.

\subsection*{Trisections}

Another construction of 4-manifolds that has close ties with the examples of this paper is based on decomposing $M_4$ into three basic pieces along a genus-$g$ central surface $\Sigma_g$. Such trisections, analogous to Heegaard decompositions in dimension 3, allow to construct an arbitrary smooth 4-manifold and the initial steps of the corresponding computation of $\ZVW (M_4)$ based on this technique were discussed in \cite{MR4047476}.

\subsection*{Variants}

It is standard in gauge theory that, depending on how certain analytical details are treated, one can obtain different versions of the homological invariant. For example, the standard definition of the $Q$-cohomology in Donaldson-Witten and in Vafa-Witten theories leads to $HF (M_3)$ and $\HVW (M_3)$, such that both are isomorphic to $\CT^+$ for $M_3 = S^3$, in particular illustrating \eqref{HFinHVW}.
This should not be confused with a finite-dimensional version of the instanton Floer homology, $I_* (M_3)$, such that {\it e.g.} for the Poincar\'e sphere $P = \Sigma (2,3,5)$:
\be
I_n (P) \; = \;
\begin{cases}
\Z , & \text{if } n=0,4 \\
0 , & \text{otherwise}
\end{cases}
\label{IPoincare}
\ee
Here, $n$ is the mod 8 grading, and the two generators in degree $n=0$ and 4 correspond to the two irreducible representations $\pi_1 (P) \to SU(2)$. (See {\it e.g.} \cite{MR1051101} for more details.) Another variant is the framed instanton homology, $I^{\#} (M_3) := I (M_3 \# T^3)$, that was already mentioned in \eqref{framedinstanton}.

Similarly, in the context of the Vafa-Witten theory the computation of $Q$-cohomology naturally leads to $\HVW (M_3)$, which is expected to be infinite-dimensional for any 3-manifold, {\it cf.} Conjecture~\ref{conj:infdimnl}. In particular, as Proposition~\ref{prop:HS3} illustrates, we expect an isolated reducible representation $\pi_1 (M_3) \to SL(2,\C)$ to contribute to $\HVW (M_3)$ a tower of states $\CT^+$, as in (\ref{HFbasic}a). Moreover, the discussion in section \ref{sec:Qcohomology} makes it clear that isolated irreducible representations $\pi_1 (M_3) \to SL(2,\C)$ (modulo conjugation) also contribute to $\HVW (M_3)$, though their contribution may be finite-dimensional, as it also happens in the monopole Floer homology or $HF^+ (M_3)$.

For example, $M_3 = \Sigma (2,3,7)$ has a total of four $SL(2,\C)$ flat connections, one of which is trivial, while the other three correspond to irreducible representations $\pi_1 (M_3) \to SL(2,\C)$, modulo conjugation. Therefore, we expect four contributions to $\HVW (\Sigma (2,3,7))$: one copy of $\CT^+$ as in the case of $M_3 = S^3$, and three finite-dimensional contributions due to irreducible $SL(2,\C)$ flat connections.
The instanton homology $I_* (\Sigma (2,3,7))$ has a similar structure, except that it does not have a contribution of a trivial flat connection --- something we already saw in \eqref{IPoincare} --- and instead of three contributions from irreducible flat connections it has only two:
\be
I_n (\Sigma (2,3,7)) \; = \;
\begin{cases}
\Z , & \text{if } n=2,6 \\
0 , & \text{otherwise}
\end{cases}
\ee
The reason for this is that two out of three irreducible $SL(2,\C)$ flat connections can be conjugated to $SU(2)$, {\it i.e.} they correspond to irreducible representations $\pi_1 (M_3) \to SU(2)$, whereas the last one is not.
On the other hand, much like $\HVW (M_3)$, a sheaf-theoretic model for $SL(2,\C)$ Floer homology \cite{MR4167016} receives contributions from all three irreducible flat connections on $M_3 = \Sigma (2,3,7)$,
\be
HP^* (\Sigma (2,3,7)) = \Z_{(0)} \oplus \Z_{(0)} \oplus \Z_{(0)}
\ee
and a similar consideration for other Brieskorn spheres suggests that, at least in this family, we may expect an isomorphism
\be
\HVW^* (\Sigma (p,q,r)) \; \stackrel{?}{\cong} \; \CT^+ \oplus HP^* (\Sigma (p,q,r))
\ee
As a natural direction for future work, it would be interesting to either prove or disprove this relation (and, in the former case, generalize to more general 3-manifolds, such as plumbings mentioned earlier, {\it cf.} Figure~\ref{fig:plumbing}).

\subsection*{Towards surgery formulae in Vafa-Witten theory}

There are two types of surgery formulae that we can consider: surgeries in three dimensions relevant to the computation of $\HVW (M_3)$, and surgeries in four dimensions that produce $\ZVW (M_4)$ via cutting-and-gluing along 3-manifolds.

The infinite family of 3-manifolds considered in this paper has some direct connections to notable surgery operations in four dimensions. Thus, $M_3 = S^2 \times S^1$ is relevant to the Gluck twist, whereas $M_3 = T^3$ is relevant to knot surgery and the log-transform. All these surgery operations consist of cutting a 4-manifold along the corresponding $M_3$ and then re-gluing pieces back in a new way, or gluing in new four-dimensional pieces with the same boundary $M_3$. This again requires understanding how a given element of the mapping class group, $\varphi \in \mathrm{MCG} (M_3)$, acts on $\HVW (M_3)$, {\it cf.} \eqref{MCGaction}.
For example, our preliminary analysis indicates that the Gluck involution that generates $\Z_2 \subset \mathrm{MCG} (S^2 \times S^1)$, associated with $\pi_1 (SO(3)) = \Z_2$, acts trivially on $\HVW (S^2 \times S^1)$.
This implies that $\ZVW (M_4, G)$ can not detect the Gluck twist when $\pi_1 (G) = 1$. (The last condition appears in view of Remark~\ref{rem:decoration}.)
We plan to return to this in future work and also analyze other important elements of $\mathrm{MCG} (M_3)$ acting on $\HVW (M_3)$.

In the case of gluing along $M_3 = T^3$, applying \eqref{VWKahler} to a family of elliptic fibrations $M_4 = E(n)$ with $\chi = 12n$, $\sigma = - 8n$, and $n-1$ basic SW classes, we quickly learn that \eqref{VWKahler} can not be consistent with a simple multiplicative gluing formula {\it a la} \cite{MR1492130,MR1833751}. Indeed, representing $E(n)$ as an iterated fiber sum, we have
\be
E(n) \; = \; \big( E(n-2) \setminus N_F \big) \cup_{T^3} \big( E(2) \setminus N_F \big)
\label{Enfibersum}
\ee
where $N_F \cong T^2 \times D^2$ is a neighborhood of a generic fiber. Then, assuming {\it e.g.} $n = \text{even}$, a simple multiplicative gluing formula along $M_3 = T^3$ would imply
\be
\ZVW (E(n)) = \left( \frac{\ZVW (E(4))}{\ZVW (E(2))} \right)^{\frac{n-2}{2}} \ZVW (E(2))
\label{Enproduct}
\ee
which is not the case:

\begin{proposition}
\label{prop:En}
For $M_4 = E(n)$ with $n = \text{even}$, \eqref{VWKahler} gives
\be
\label{ZVWEnshort}
\ZVW (E(n)) \; = \;
\begin{cases}
(-1)^{\frac{n}{2} + 1}
{n -2 \choose \frac{n}{2} -1}
\frac{1}{2}
\left( \frac{G(q^2)}{4} \right)^{\frac{n}{2}} , & \text{if } \; n > 2 \\
\frac{1}{8} G(q^2) + 
\frac{1}{4} G(q^{1/2}) + 
\frac{1}{4} G(-q^{1/2}) , & \text{if } \; n=2
\end{cases}
\ee
where
$$
G(q) = \frac{1}{\eta^{24}} = \frac{(2\pi)^{12}}{\Delta (\tau)}
= \frac{1}{q} \Big(
1
+24 q
+324 q^2
+3200 q^3
+25650 q^4
+ \ldots \Big)
$$
\end{proposition}

\begin{proof}
The elliptic surface $E(n)$ has the following basic topological invariants
\be
b_1=0
\,, \qquad
2 \chi + 3 \sigma = 0
\,, \qquad
\frac{\chi + \sigma}{4} = n
\ee
Moreover, $F \cdot F = 0$ and the basic classes are $(n-2j) F$, with $j = 1, \ldots, n-1$ and the corresponding Seiberg-Witten invariants:
\be
\mathrm{SW}_{E(n)} ({\mathfrak{s}}_j) \; = \; (-1)^{j+1} {n -2 \choose j-1}
\ee
In order to obtain the $G=SU(2)$ invariant of $M_4 = E(n)$, we need to substitute all these into \eqref{VWKahler}, evaluate it at $v=0$ and also multiply by a factor of $\frac{1}{2}$ (associated with the center of $SU(2)$).
For $n = \text{even}$, a straightforward calculation gives
\begin{multline}
\label{ZVWEnlong}
\ZVW (E(n)) = (-1)^{\frac{n}{2} + 1}
{n -2 \choose \frac{n}{2} -1}
\frac{1}{2}
\left( \frac{G(q^2)}{4} \right)^{\frac{n}{2}} + 
\\
+ \sum_{j=1}^{n-1} (-1)^{j+1} {n -2 \choose j-1}
\Big[
\left( \frac{G(q^{1/2})}{4} \right)^{\frac{n}{2}}
+ \left( \frac{G(-q^{1/2})}{4} \right)^{\frac{n}{2}}
\Big]
\end{multline}
By the binomial formula, the second line in this expression is non-zero only for $n=2$ and, therefore, we otain \eqref{ZVWEnshort}.
\end{proof}

This result is not surprising because gluing along $M_3$ is expected \cite{MR4047476} to be governed by $\text{MTC} [M_3]$, which is very non-trivial for $M_3 = T^3$. Indeed, modulo delicate questions related to KK modes (which will be discussed elsewhere), the calculations in section~\ref{sec:eqVerlinde} can be interpreted as the computation of \eqref{KMTCss} for $M_3 = T^3$. Indeed, since $M_4 = T^2 \times \Sigma_g$ can be obtained by gluing basic pieces\footnote{These basic pieces are products of $T^2$ with pairs-of-pants, illustrated in Figure~\ref{fig:pants}.} along 3-tori and the corresponding calculation of Vafa-Witten invariants for $G=SU(2)$ is expressed as a sum \eqref{ZviaSmatrix} where $\lambda$ takes 10 possible values, {\it cf.} \eqref{Selements}, we expect the gluing formula for \eqref{Enfibersum} to be also a sum over the same set of $\lambda$. In comparison, the naive multiplicative gluing formula {\it a la} \eqref{Enproduct} would mean that the sum over $\lambda$ consists of a single term.

Surgery operations that involve cutting and gluing $M_3$ itself, {\it i.e.} surgery formulae for $\HVW (M_3)$, are also interesting. For example, continuing the parallel with the Donaldson-Witten theory, one might expect the standard surgery exact triangles
\be
\ldots \xrightarrow[~~~~]{~~\phantom{i}~~} \HVW (S^3) \xrightarrow[~~~~]{~~\phantom{i}~~} \HVW (S^3_0 (K)) \xrightarrow[~~~~]{~~\phantom{i}~~} \HVW (S^3_{+1} (K)) \xrightarrow[~~~~]{~~\phantom{i}~~} \ldots
\label{HVWS3triangle}
\ee
or, more generally,
\be
\ldots \xrightarrow[~~~~]{~~\phantom{i}~~} \HVW (Y_0 (K))  \xrightarrow[~~~~]{~~i~~} \HVW (Y_{+1} (K)) \xrightarrow[~~~~]{~~j~~} \HVW (Y) \xrightarrow[~~~~]{~~k~~} \HVW (Y_0 (K)) \xrightarrow[~~~~]{~~\phantom{i}~~} \ldots
\label{HVWtriangle}
\ee
where $K$ is a knot in a homology 3-sphere $Y$. Although such surgery exact triangles are ubiquitous in gauge theory --- apart from the original Floer homology, they also exist in many variants of the monopole Floer homology and its close cousin, the Heegaard Floer homology --- it is not clear whether they hold in Vafa-Witten theory. One important difference that already entered our analysis in a number of places has to do with degree shifts.
Indeed, all maps in \eqref{HVWtriangle} and, similarly, in the oppositely-oriented exact sequence associated with the inverse surgery operation,
\be
\ldots \xrightarrow[~~~~]{~~\phantom{i}~~} \HVW (Y_{-1} (K)) \xrightarrow[~~~~]{~~\phantom{i}~~} \HVW (Y_0 (K)) \xrightarrow[~~~~]{~~\phantom{i}~~} \HVW (Y) \xrightarrow[~~~~]{~~\phantom{i}~~} \HVW (Y_{-1} (K)) \xrightarrow[~~~~]{~~\phantom{i}~~} \ldots
\label{HVWopptriangle}
\ee
are induced by oriented cobordisms. For exampe, the map $j$ in \eqref{HVWtriangle} is induced by a cobordism $W$ with $b_1 (W) = b_2^+ (W) = 0$ and $b_2^- (W) = 1$, {\it cf.} \cite{MR1362829}. Given such topological data of a cobordism, the degree shift of the corresponding map can be computed by evaluating the virtual dimension (``ghost number'' anomaly) of a given theory on $W$. However, the balanced property of the Vafa-Witten theory that appeared a number of times earlier makes this quantity vanish for any $W$, so that all maps in \eqref{HVWtriangle} and \eqref{HVWopptriangle} have degree zero.

Another, perhaps even more important feature of the Vafa-Witten theory in regard to the existence of surgery exact triangles directly follows from the computations in this paper. Namely, earlier we saw that $\HVW (S^2 \times S^1)$ is much larger than $\HVW (S^3)$: while the latter is isomorphic to $\CT^+$, the former has GK-dimension 4 ({\it i.e.} looks like $(\CT^+)^{\otimes 4}$). The reason $M_3 = S^3$ and $S^2 \times S^1$ are relevant to the question about surgery exact triangles is that they tell us about all terms in \eqref{HVWS3triangle} when $K=\mathrm{unknot}$: in this csae $S^3_{+1} (K) \cong S^3$ and $S^3_0 (K) \cong S^2 \times S^1$. It is not clear how such a sequence could be exact, suggesting that the standard form of the surgery exact triangles may not hold in the Vafa-Witten theory. One possible scenario is that a spectral sequence as in section~\ref{sec:Qcohomology} can reduce the size of $\HVW (S^2 \times S^1)$ to roughly twice the size of $\HVW (S^3)$, thus providing a more natural definition of $\HVW (M_3)$ in view of \eqref{HVWS3triangle}.

This example, with $K=\mathrm{unknot}$, also illustrates well the origin of the problem: the reason $\HVW (S^2 \times S^1)$ is much larger than $\HVW (S^3)$ has to do with non-compactness of the moduli spaces, as we saw earlier in section~\ref{sec:eqVerlinde} and, therefore, suggests that addressing the non-compactness of the moduli spaces may help with the surgery exact triangles.
It would be interesting to shed more light on this question.

\appendix

\section{Supersymmetry algebra}
\label{sec:appendix}

In addition to scalar (0-form) supercharges \eqref{scalarQ}, the topological Vafa-Witten theory also has vector (1-form) supercharges \cite{Geyer:2001yc}:
\be
\begin{aligned}
\bar Q^a_{\mu} A_{\nu} & = \delta_{\mu \nu} \eta^a + \chi_{\mu \nu}^a \\
\bar Q^a_{\mu} B_{\rho \sigma}^+ & = - \delta_{\mu [ \rho} \psi_{\sigma]}^a - \epsilon_{\mu \nu \rho \sigma} \psi^{\nu a} \\
\bar Q^a_{\mu} \phi^{bc} & = - \tfrac{1}{2} \epsilon^{ab} \psi^c_{\mu} - \tfrac{1}{2} \epsilon^{ac} \psi^b_{\mu} \\
\bar Q^a_{\mu} \eta^b & = D_{\mu} \phi^{ab} + \epsilon^{ab} H_{\mu} \\
\bar Q^a_{\mu} \psi_{\nu}^b & = - \epsilon^{ab} F_{\mu \nu} + \delta_{\mu \nu} \epsilon_{cd} [\phi^{ac} , \phi^{bd}] + \epsilon^{ab} G^+_{\mu \nu} - [B^+_{\mu \nu} , \phi^{ab}] \\
\bar Q^a_{\mu} H_{\nu} & = D_{\mu} \psi^a_{\nu} - \tfrac{1}{2} D_{\nu} \psi^a_{\mu} + \epsilon_{cd} [\phi^{ac} , \chi^d_{\mu \nu} - \delta_{\mu \nu} \eta^d] + [B^+_{\mu \nu} , \eta^a] \\
\bar Q^a_{\mu} \chi^b_{\rho \sigma} & = \delta_{\mu [ \rho} D_{\sigma ]} \phi^{ab} + \epsilon_{\mu \nu \rho \sigma} D^{\nu} \phi^{ab} - \epsilon^{ab} \delta_{\mu [ \rho} H_{\sigma ]} - \epsilon^{ab} \epsilon_{\mu \nu \rho \sigma} H^{\nu} - \epsilon^{ab} D_{\mu} B^+_{\rho \sigma} \\
\bar Q^a_{\mu} G^+_{\rho \sigma} & = D_{\mu} \chi^a_{\rho \sigma} - \delta_{\mu [ \rho} D_{\sigma ]} \eta^a - \epsilon_{\mu \nu \rho \sigma} D^{\nu} \eta^a - \epsilon_{cd} [\phi^{ac} , \delta_{\mu [ \rho} , \psi^d_{\sigma ]} + \epsilon_{\mu \nu \rho \sigma} \psi^{\nu d} ] + \tfrac{1}{2} [\psi^a_{\mu} , B^+_{\rho \sigma} ]
\end{aligned}
\notag
\ee
The algebra of these supercharges, up to gauge transformations, is
\begin{eqnarray}
\{ Q^a , Q^b \} & = & 0 \notag \\
\{ Q^a , \bar Q^b_{\mu} \} & = & - \epsilon^{ab} \partial_{\mu} \\
\{ \bar Q^a_{\mu} , \bar Q^b_{\nu} \} & = & 0 \notag
\end{eqnarray}
For $\mu = 0$, it should be compared with the supersymmetry algebra in quantum mechanics. With two supercharges, the supersymmetry algebra in quantum mechanics looks like $\{ Q, Q^{\dagger} \} = 2H$. With the enhanced supersymmetry, it has the form
\be
\{ Q_i , Q_j^{\dagger} \} = 2 \delta_{ij} H
\,, \qquad
\{ Q_i , Q_j \} = 0
\ee

In the main text we also encountered a closely related 2d $\CN=(2,2)$ supersymmetry algebra:
\be
\{ Q_+ , \bar Q_+ \} = \tfrac{1}{2} (H+P) = H_L
\,, \qquad
\{ Q_- , \bar Q_- \} = \tfrac{1}{2} (H-P) = H_R
\ee
and the topological A-model, that involves linear combinations of the above supercharges, $Q_A = Q_+ + \bar Q_-$ and $\bar Q_A = \bar Q_+ + Q_-$, such that $\{ Q_A , \bar Q_A \} = H$ and $Q_A^2 = 0$.


\newpage
\bibliography{VWrefs}
\bibliographystyle{alpha}

\end{document}